\documentclass[12pt, twoside, leqno]{amsart}
\usepackage{amsmath,amsthm}
\usepackage{amssymb,latexsym, amscd}
\usepackage{enumerate}
\usepackage{color}
\usepackage{bm}
\usepackage[all]{xy}
\usepackage{amscd}

\usepackage{tikz}
\usetikzlibrary{cd}

\usepackage{latexsym, amscd}
\usepackage{enumerate}
\usepackage{color}
\usepackage{bm}
\usepackage[all]{xy}

\newcommand{\Z}{\mathbb{Z}}
\newcommand{\N}{\mathbb{N}}
\newcommand{\bN}{\mathbb{N}}
\newcommand{\Q}{\mathbb{Q}}
\newcommand{\R}{\mathbb{R}}
\newcommand{\F}{\mathbb{F}}
\newcommand{\Zp}{\mathbb{Z}_p}

\newcommand{\cC}{\mathcal{C}}

\newcommand{\cL}{\mathcal{L}}
\newcommand{\cO}{\mathcal{O}}
\newcommand{\cP}{\mathcal{P}}

\newcommand{\cZ}{\mathcal{Z}}

\newcommand{\Cmodsim}{\cC/\text{\lower.5ex\hbox{$\sim$}}\,\, }
\newcommand{\Cpmodsim}{\cC_p/\text{\lower.5ex\hbox{$\sim$}}\,\, }
\newcommand{\modsim}{\text{\lower.5ex\hbox{$\sim$}}\,\, }

\newcommand{\Lp}{L} 
\newcommand{\cLp}{\cL} 

\newcommand{\wtil}[1]{\widetilde{#1}}
\newcommand{\ol}[1]{\overline{#1}}

\numberwithin{equation}{section}

\DeclareMathOperator{\Gal}{Gal}

\DeclareMathOperator{\Cok}{Cok}

\DeclareMathOperator{\Cl}{Cl}

\DeclareMathOperator{\Lat}{Lat}

\DeclareMathOperator{\pd}{pd}
\DeclareMathOperator{\Hom}{Hom}
\DeclareMathOperator{\Ext}{Ext}

\DeclareMathOperator{\Frac}{Frac}
\DeclareMathOperator{\rank}{rank}

\DeclareMathOperator{\pe}{pe}
\DeclareMathOperator{\adm}{adm}
\DeclareMathOperator{\rea}{real}

\DeclareMathOperator{\ram}{ram}

\theoremstyle{plain}
\newtheorem{thm}{Theorem}[section]
\newtheorem{lem}[thm]{Lemma}

\newtheorem{prop}[thm]{Proposition}

\newtheorem{claim}[thm]{Claim}

\theoremstyle{definition}
\newtheorem{defn}[thm]{Definition}

\newtheorem{rem}[thm]{Remark}

\title[Class groups as Galois modules -- the nonabelian case]
 {Determining class groups as Galois modules up to equivalence for some
nonabelian extensions}

\author[C.~Greither]{Cornelius Greither}
 \address{Fakult\"at Informatik, Universit\"at der Bundeswehr M\"unchen,
 85577 Neubiberg, Germany}
 \email{cornelius.greither@unibw.de}

 \author[T.~Kataoka]{Takenori Kataoka}
\address{Department of Mathematics, Faculty of Science Division II, Tokyo University of Science.
1-3 Kagurazaka, Shinjuku-ku, Tokyo 162-8601, Japan}
\email{tkataoka@rs.tus.ac.jp}

\usepackage{mathrsfs}
\newcommand{\lrang}[1]{\langle #1 \rangle}
\newcommand{\cyc}{{\rm cyc}}
\newcommand{\cG}{\mathcal{G}}
\newcommand{\cK}{\mathcal{K}}
\newcommand{\sV}{\mathscr{V}}
\newcommand{\sW}{\mathscr{W}}
\DeclareMathOperator{\cd}{cd}
\DeclareMathOperator{\Res}{Res}
\newcommand{\Gamm}{\Gamma}

\begin{document}
\maketitle

\begin{abstract}
In previous papers, the Galois module structure of minus class groups was studied for abelian CM extensions.
In this paper, we discuss some nonabelian cases, focusing on metacyclic extensions.
For a certain class of these, we obtain a complete description of the Galois modules
that occur as minus class groups, modulo a certain equivalence relation on modules,
which was introduced earlier by the same authors.
\end{abstract}

\section{Introduction}

Let $L/K$ be a finite Galois CM extension of number fields, i.e., $L$ is a CM-field and $K$ is a totally real field.
Let $\Cl_L^{T}$ be the $T$-ray class group of $L$, where $T$ 
is an auxiliary finite set of primes.
We study the Galois module structure of $\Cl_L^T$ or, rather, its minus part $\Cl_L^{T, -}$ with respect to the complex conjugation $j \in \Gal(L/K)$.

\subsection{Review of the abelian case}

In the paper \cite{GK} a notion of equivalence for modules over
a commutative group ring was developed.
This was then applied in \cite{GK2} to the study of $\Cl_L^{T, -}$ for abelian CM extensions $L/K$.

To be precise, let $\Gamma$ be a finite abelian group. 
Let $\cC$ be the category of finite $\Z'[\Gamma]$-modules, where we put $\Z' = \Z[1/2]$.
In \cite{GK} we introduced an equivalence relation $\sim$ on $\cC$.
The direct sum $\oplus$ defines a commutative monoid structure on the set $\Cmodsim$.

For each abelian CM extension $L/K$ such that $\Gal(L/K) \simeq \Gamma \times \{1, j\}$, we have a natural identification $\Z[\Gal(L/K)]^- \simeq \Z'[\Gamma]$.
Therefore, we can talk about the class of $\Cl_L^{T, -}$ in $\Cmodsim$.
We call an element of $\Cmodsim$ a realizable class if it is the class of $\Cl_L^{T, -}$ for some $L/K$.
Let $\cZ^{\rea} \subset \Cmodsim$ be the subset of realizable classes.

To study the structure of $\cZ^{\rea}$, in \cite{GK2} we introduced 
 a submonoid $\cZ^{\adm} \subset \Cmodsim$, whose elements are said to be admissible.
It is defined so that $\cZ^{\rea} \subset \cZ^{\adm}$, which means that we have a 
necessary condition for  realizability; roughly speaking, the restrictions are given by ramification data.
Indeed, we showed that typically only a small proportion
of all finite modules can appear, up to equivalence,
as a module $\Cl_L^{T, -}$.
This is made quantitative in a way by demonstrating
how small $\cZ^{\adm}$ is in $\Cmodsim$ (see the discussion after \cite[Corollary 1.4]{GK2}).

Moreover, we also proved some constructive results:
certain modules can indeed be realized as $\Cl_L^{T, -}$, up to equivalence,
by some extension $L/K$. 
This is formulated as $\cZ^{\rea} = \cZ^{\adm}$ in some cases (\cite[Theorem 1.2]{GK2}).

\subsection{Aim of this paper}

We now propose to discuss some situations where
$L/K$ is no longer assumed to be abelian. 
Again we are
interested in a priori
restrictions on $\Cl_L^{T, -}$ as a Galois module. We would also like to explore 
constructive results, and for some small
nonabelian groups we are able to do this.

First in \S \ref{sec:alg}, we show that the algebraic theory of equivalence of modules can be generalized to 
the nonabelian case in a straightforward way.
Thus, for a finite group $\Gamma$, we again have the monoid $\Cmodsim$ of classes of finite $\Z'[\Gamma]$-modules.
We will also show that as in the abelian case there is a natural injective homomorphism $\Phi: \Cmodsim \hookrightarrow \Lat^{\pe}_{\Z'[\Gamma]}$, where $\Lat^{\pe}_{\Z'[\Gamma]}$ denotes a certain monoid of lattices.
This enables us to study finite modules via lattices.

In \S \ref{sec:arith}, we define the subset of realizable classes $\cZ^{\rea} \subset \Cmodsim$ in the same way as in the abelian case.
We then show that the key exact sequence involving $\Cl_L^{T, -}$, which determines the class of $\Cl_L^{T, -}$ in $\Cmodsim$,
generalizes as well. That sequence goes back
to work \cite{AK} of Atsuta and the second author, and is in a way inherent
in \cite{GK2}. 
By this we can define $\cZ^{\adm} \subset \Cmodsim$ such that $\cZ^{\rea} \subset \cZ^{\adm}$ a priori.

Discussions of this kind of generalization tend to be dull and un-enlightening
if all details are spelled out; but on the other hand, a certain minimum
of detail should be given in order to be convincing. We hope
to have made reasonable choices in this respect.

\subsection{Structure of admissible classes}

We now fix an odd prime $p$.
For a Galois CM extension $L/K$ with $\Gal(L/K) \simeq \Gamma \times \{1, j\}$, we study the $p$-part ${}_p \Cl_L^{T, -}$ of $\Cl_L^{T, -}$ as a $\Z_p[\Gamma]$-module.
Instead of $\Cmodsim$, we naturally consider the monoid $\Cpmodsim$ of finite $\Z_p[\Gamma]$-modules up to equivalence.
We define $\cZ_p^{\rea} \subset \cZ_p^{\adm} \subset \Cpmodsim$ in a similar manner.
Taking the $p$-parts simplifies several things.

In \S \ref{sec:metaab}, we will discuss
an explicit class of nonabelian groups
\[
\Gamm = C_p \rtimes C_r,
\]
where $C_p$ and $C_r$ are cyclic groups of order $p$ and $r$ respectively, and $C_r$ acts on $C_p$ faithfully (so $r$ is a divisor of $p-1$).
The prime $p$ is here the same as above, that is, we will be
looking at the $p$-adified situation. 
We do not assume that $r$ is a prime.
A prototypical
case is $\Gamm=S_3$ with $p = 3$ and $r = 2$. 

As the first main result, we will determine the structure of the monoid $\cZ_p^{\adm}$ of admissible classes.
To do this, we first determine the structure of $\Cpmodsim$ by using the classification of
indecomposable lattices over $\Z_p[\Gamma]$, which has been known for a long time.
We will then see
that only a rather small proportion
is admissible:

\begin{thm}\label{thm:11}
Let $\Gamm = C_p \rtimes C_r$.
\begin{itemize}
\item[(1)]
The monoid $\Cpmodsim$ is not free unless $r = 1$, and it generates a free abelian group of rank $2r-1$.
\item[(2)]
The submonoid $\cZ_p^{\adm}$ is a free monoid of rank $\sigma_0(r)$, 
the number of positive divisors of $r$.
\end{itemize}
\end{thm}

If $r > 1$, then the rank $\sigma_0(r)$ of $\cZ_p^{\adm}$ is strictly smaller than the rank $2r-1$ of $\Cpmodsim$.

\subsection{Realizability problem}

Finally we discuss the realizability problem.
In \S \ref{sec:realize}, we will handle a rather wide class of finite groups.
As a special case, in \S \ref{sec:rea_metacyc},
we will prove for the class of
groups $\Gamma = C_p \rtimes C_r$ that 
all admissible classes
are in fact realizable:

\begin{thm}\label{thm:12}
For $\Gamm = C_p \rtimes C_r$, 
we have $\cZ_p^{\rea} = \cZ_p^{\adm}$.
\end{thm}

This means that all admissible $\Z_p[\Gamma]$-modules do arise, up to equivalence, as the module ${}_p \Cl_L^{T, -}$ coming from
some CM extension of number fields having Galois group $ \Gamm\times\{1,j\}$.
Upon combining Theorems \ref{thm:11} and \ref{thm:12}, we  
arrive at a complete description  of $\cZ_p^{\rea}$ for $\Gamma = C_p \rtimes C_r$.

\section{Algebraic preliminaries}\label{sec:alg}

We start by recalling the definition
of equivalence of modules from \cite[\S 2]{GK}, 
generalizing it to the nonabelian case at once.
For the sake of simplicity we restrict the scope somewhat. Let $R$ be a ring of the
shape $\mathcal O[\Gamma]$, where $\mathcal O$ is $\Z$, $\Z' = \Z[1/2]$, or $\Z_p$ 
for some prime number $p$, and $\Gamma$ is a finite group.
All modules are on the left.


Let $\mathcal C = \cC_R$ be the category of all finite $R$-modules. 
Let $\mathcal P = \cP_R$ denote the subcategory of $\cC$ consisting of modules of projective dimension at most one ($ \pd \le1$ for short).
Note that for any module, having $\pd \le 1$ is equivalent to 
being $\Gamma$-cohomologically trivial in this case.

The following definition makes perfect sense without
any conceivable ambiguity if $\Gamma$ is not assumed abelian and we use 
left modules throughout. 

\begin{defn}
Two finite $R$-modules $X$ and $Y$
are called equivalent ($X \sim Y$) if there are finite
$R$-modules $M$ and $N$ provided with
three-step filtrations $0 \subset M' \subset M'' \subset M$ and
$0 \subset N' \subset N'' \subset N$ satisfying
\begin{itemize}
\item[(i)]
 $M$ and $N$
are $R$-isomorphic;
\item[(ii)]
 the top and bottom layers $M'$, $N'$,
$M/M''$, $N/N''$ all have $\pd \le1$ over $R$; and 
\item[(iii)]
the middle layer $M''/M'$ is
isomorphic to $X$, and the middle layer $N''/N'$ is
isomorphic to $Y$.
\end{itemize}
Then $\sim$ is indeed an equivalence relation on the objects of $\cC$.
The set $\Cmodsim$ of equivalence classes has a monoid structure induced by the direct sum 
operator $\oplus$.
\end{defn}

It is an important observation that finite modules up to equivalence can be described by lattices up to projective equivalence:

\begin{defn}
A module $L$ over $R = \cO[\Gamma]$ is called a lattice if it is finitely generated,
and without $\cO$-torsion. 
Two lattices $L$ and $L'$ are called projectively
equivalent, if there exist finitely generated projective $R$-modules
 $F$ and $F'$
such that $L \oplus F' \cong L' \oplus F$; this is an equivalence relation.
The set
of equivalence classes modulo projective equivalence is denoted
by $\Lat^{\pe} = \Lat^{\pe}_R$.
The set $\Lat^{\pe}$ has a commutative monoid structure by the direct sum $\oplus$.
\end{defn}

\begin{prop}\label{prop:fin_lat}
We have a well-defined injective homomorphism
\[
\Phi: \Cmodsim \hookrightarrow \Lat^{\pe}
\]
defined as follows:
For a module $X$ in $\mathcal C$, 
we choose a surjection $F \to X$ from a projective
lattice $F$, and define $\Phi(X)$ as its kernel.

Moreover, the image of $\Phi$ is exactly the set of classes of $L$ such that 
$\Frac(\cO) \otimes_{\cO} L \simeq \Frac(\cO) \otimes_{\cO} F$ for some projective $R$-module $F$, where $\Frac(\cO)$ denotes the field of fractions of $\cO$.
\end{prop}

\begin{proof}
In the abelian case this is shown in \cite[Theorem 4.2]{GK}.
In the proof, we used a nonzerodivisor $z$ that annihilates a module 
$X$ over  $R$. In the nonabelian context, we need to choose
$z$ central, but this is easily done; simply take a nonzero element of $\cO$
that annihilates the finite module $X$. The rest of the argument is
more or less unchanged. Likewise, the nonzerodivisor $f$ in the proof
of \cite[Proposition 2.5]{GK}, which is used in  \cite[Theorem 4.2]{GK}, should be chosen central.
   The statement about the image of $\Phi$ follows from the construction:
    if $X\in \mathcal C$ and $0 \to F' \to F \to X \to 0$ is exact with $F$ projective,
    then as $X$ is finite, $F'$ and $F$ become isomorphic after applying $\Frac(\cO) \otimes_{\cO} (-)$.
The converse is shown in a similar way.
\end{proof}

\begin{rem}\label{rem:Latpe_free}
When $\cO = \Z_p$, we have the Krull--Schmidt theorem:
Any lattice can be written as a direct sum of indecomposable lattices in a unique way.
Indeed, all indecomposable lattices have local 
endomorphism rings. (The reason for this is that we are dealing with
$p$-adically complete algebras, and therefore idempotents can
always be lifted; the endomorphism ring of an indecomposable module
contains no nontrivial idempotents, so it must be local.)
It then follows that the commutative monoid $\Lat^{\pe}$ is free on the set of non-projective, indecomposable lattices modulo isomorphism.
\end{rem}

We will need that one main technical device remains operational.
This  is a kind of first syzygy operator $\omega^1$, also called the shift operator.

\begin{prop}
We have a well-defined automorphism 
\[
\omega^1: \Cmodsim \overset{\simeq}{\to} \Cmodsim
\]
defined as follows:
For a module $X$ in $\cC$, we choose a surjection $P \to X$ from a module $P$ in $\cP$,
and define $\omega^1(X)$ as its kernel.
\end{prop}

\begin{proof}
In the abelian case this is shown in \cite[Definition 3.4]{GK} (where the notation
$X[1]$ is used instead of $\omega^1(X)$) and the following discussion, which 
relies on \cite[Theorem 3.19]{Ka}. The arguments in that definition
carry over without difficulty to the nonabelian case. 
\end{proof}

It is important to understand how $\Phi$ behaves with respect
to the first syzygy operator $\omega^1$.
This works exactly as in \cite[Lemma 2.4]{GK2}.

\begin{prop}\label{prop:omegaO}
We have a well-defined automorphism
\[
\Omega: \Lat^{\pe} \overset{\simeq}{\to} \Lat^{\pe}
\]
defined as follows:
For a lattice $L$, we choose a surjection $F \to L$ from a projective lattice $F$, and define $\Omega(L)$ as its kernel.
We have a commutative diagram
\[
\xymatrix{
	\Cmodsim \ar[r]^{\omega^1}_{\simeq} \ar@{^(->}[d]_{\Phi}
	& \Cmodsim \ar@{^(->}[d]^{\Phi}\\
	\Lat^{\pe} \ar[r]_{\Omega}^{\simeq}
	& \Lat^{\pe}.
	}
\]
\end{prop}

\section{Realizability and admissibility}\label{sec:arith}

\subsection{Realizability}

Let $L/K$ be a finite Galois CM extension with Galois group $G$;
the element $j\in G$ will always denote complex conjugation. 
Let $T$ be a finite set of finite primes of $K$ such that 
\begin{itemize}
\item
$L/K$ is unramified at any $v \in T$ and
\item
$L$ has
no nontrivial roots of unity congruent to 1 at all $w \in T_L$, where $T_L$ denotes the set of places of $L$ above
places $v \in T$. 
\end{itemize}
We do not 
reproduce the definition of the $T$-modified
class group $\Cl_L^{T}$ here, but instead we refer to \cite[\S 2.4]{GK2}.
We consider its minus part $\Cl_L^{T, -}$ as a finite module over $\Z[G]^- = \Z'[G]/(1+j)$.

The influence of the set $T$ is rather limited; often it may be disregarded by the following lemma (see \cite[Lemma 2.5]{GK2}).

\begin{lem}
The following hold.
\begin{itemize}
\item[(1)]
The equivalence class of $\Cl_L^{T,-}$ as a finite $\Z[G]^-$-module
is independent of $T$ as long as it satisfies the hypothesis described above.
\item[(2)]
Let $p$ denote an odd prime.
If $L$ does not contain $\mu_p$, then the $p$-parts of the two modules 
$\Cl_L^{-}$ and $\Cl_L^{T,-}$ are equivalent as finite $\Z_p[G]^-$-modules.
\end{itemize}
\end{lem}

\begin{proof}
Let $\kappa(w)$ denote the residue field of $\cO_L$ at a finite prime $w$ of $L$.
Then we have an exact sequence
\[
\cO_L^{\times} \to \prod_{w \in T_L} \kappa(w)^{\times} \to \Cl_L^{T}  \to \Cl_L \to 0
\]
given by class field theory.
The module $\prod_{w \in T_L} \kappa(w)^{\times}$ is $G$-cohomologically trivial since  $T$ contains
no ramified primes.
Upon taking the minus parts, and using the second condition on $T$, we obtain an exact sequence
\[
0 \to \mu(L)^- \to \left(\prod_{w \in T_L} \kappa(w)^{\times}\right)^- \to \Cl_L^{T, -}  \to \Cl_L^- \to 0,
\]
where $\mu(L)$ denotes the group of roots of unity of $L$.
Both claims (1) and (2) follow from this sequence.
 \end{proof}

\begin{defn}\label{defn:rea}
Let now $\Gamma$ be a finite group and consider the group $G = \Gamma \times \{1, j\}$, 
where $j$ is an element of exact order $2$.
We have a natural isomorphism $\Z[G]^- \simeq \Z'[\Gamma]$.
Let us call a finite $\Z'[\Gamma]$-module $X$ ``realizable'' if there exists a Galois CM
extension $L/K$ with group $G$ and complex conjugation $j$ such that $X$ is equivalent to $\Cl_L^{T, -}$.
The subset of realizable classes is denoted by $\cZ^{\rea} \subset \Cmodsim$.
In the same way, we define $\cZ_p^{\rea} \subset \Cpmodsim$, considering $\Z_p[\Gamma]$-modules instead of $\Z'[\Gamma]$-modules.
\end{defn}

\subsection{The exact sequence}

We describe a result in \cite[\S 1.1.2]{AK2}  that is proved  for the case that $G$ is abelian.
Let $L/K$ be a finite Galois CM extension with Galois group $G$;
the element $j\in G$ will always denote complex conjugation. 

For comparison with \cite{AK2}, we consider a 
slightly more general situation:
The symbols $\Sigma$ and $\Sigma'$ always denote finite sets of places of $K$ satisfying the following.
\begin{itemize}
\item[(H1)]
$\Sigma \cap \Sigma' = \emptyset$.
\item[(H2)]
$\Sigma \supset S_\infty$, the set of all infinite places.
\item[(H3)]
The only odd-order root of unity in $L$ that 
is congruent to 1 at all $w \in \Sigma'_L$
is 1.
\item[(H4)]
$\Sigma$ contains all places that ramify wildly from $K$ to $L$,
and $\Sigma \cup \Sigma'$ contains the set $S_{\ram}(L/K)$ of all ramified places. 
\end{itemize}
Note that (H4) is certainly implied by
\begin{itemize}
\item[(H4')]
$\Sigma$ contains $S_{\ram}(L/K)$. 
\end{itemize}
We will apply this to $\Sigma = S_{\infty} \cup S_{\ram}(L/K)$ and $\Sigma' = T$.

\begin{defn}
For each finite prime $v$ of $K$, after choosing a prime $w$ of $L$ above $v$, let $D_w$ 
and $I_w$ denote the decomposition and inertia groups in $G$.
Let $\varphi_w \in D_w/I_w$ be the Frobenius automorphism, which is a generator of the cyclic group $D_w/I_w$.
We define a $\Z[G]$-module $A_v$ by
\[
A_v = \Z[G] \otimes_{\Z[D_w]} \Z[D_w/I_w]/(g_w),
\quad
g_w = 1-\varphi_w^{-1}+ \# I_w \in \Z[D_w/I_w].
\]
Note that the triple $(D_w, I_w, \varphi_w)$ is replaced by a $G$-conjugate if we change the choice of $w$ above $v$.
However, the left $\Z[G]$-module $A_v$ does not change up to isomorphism, which justifies the notation.
\end{defn}

It would be suggestive  to write this again as $A_v = \Z[G/I_w]/(g_w)$
as in earlier work on the abelian case.
However this would require further explanation, because $G/I_w$ need not
exist as a group, and  it is no longer obvious either what
modding out by $g_v$ means exactly.

\begin{thm} 
One may construct a finite $\Z[G]^-$-module $\Omega_\Sigma^{\Sigma'}$ and a short exact sequence 
\[
0 \to \Cl_L^{\Sigma',-} \to \Omega_\Sigma^{\Sigma'} \to
    \bigoplus_{v\in\Sigma_f} A_v^- \to 0,
\]
where $\Sigma_f = \Sigma \setminus S_{\infty}$.
Moreover, the module $\Omega_\Sigma^{\Sigma'}$ is $G$-cohomologically trivial.
\end{thm}

\begin{proof}
In the abelian case, the construction of $\Omega_\Sigma^{\Sigma'}$ 
and the exact sequence is \cite[Proposition 1.3]{AK2}, and the $G$-cohomological triviality is \cite[Proposition 1.4]{AK2}.
Both are proved in \cite[\S 2]{AK2}.

We claim that the proof is valid even if $G$ is nonabelian.
The first four subsections \cite[\S 2.1--\S 2.4]{AK2} do not even assume that $G$ is abelian. 
Let us at least discuss the ``local'' module $W_w$ of \cite[Proposition 2.5]{AK2};
it is given as
\[
W_w = \{(x,y)\in \Delta D_w \oplus \Z[D_w/I_w] : \bar x= (1-\varphi_w^{-1})y
 \in \Z[D_w/I_w] \}.
\]
The letter $\Delta$ denotes the augmentation kernel. 
The constructions in \cite[\S 2.5]{AK2} go through for nonabelian $G$ without
problems worth mentioning. As an example, let us discuss \cite[Proposition 2.6]{AK2}, 
which establishes an exact sequence
\[
0 \to W_w \overset{f_w}{\to} \Z[D_w] \to \Z[D_w/I_w]/(g_w) \to 0,
\]
starting from the above description of $W_w$. For the proof, \cite{AK2}
refers to \cite[Lemma 3.5]{AK}.    The proof of that lemma carries over. Just a few highlights:
The $\Z[D_w]$-homomorphism $f_w$ sends a pair $(x,y)\in W_w$ to $x+\nu_{I_w} y$, where
$\nu_{I_w}$ is the norm element attached to the subgroup $I_w$.
The proof uses a module $\Cok \alpha_w$, which is claimed to be free rank 
one over $\Z[D_w/I_w]$. This is still true if
$D_w$ is not assumed abelian. Indeed the only thing that counts is
that $D_w/I_w$ is abelian; it is not necessary to use that it is 
cyclic, as written in \cite{AK}.
Then, upon performing induction from $D_w$ to $G$, we obtain another 
short exact sequence
\[
0 \to \Z[G] \otimes_{\Z[D_w]} W_w \to \Z[G] \to A_v \to 0.
\]
\end{proof}

Note that the operator $\omega^1$ obviously commutes with direct sums.
This implies in particular (\cite[Theorem 2.6]{GK2}): 

\begin{thm}\label{thm:Cl_class}
When $G = \Gamma \times \{1, j\}$, we have
\[
\Cl_L^{T,-} \sim \bigoplus_{v\in\Sigma_f} \omega^1(A_v^- )
\]
as finite modules over $\Z[G]^- \simeq \Z'[\Gamma]$.
\end{thm}

Therefore, if we know $\bigoplus_{v\in\Sigma_f} A_v^-$ up to equivalence, then we know $\Cl_L^{T,-}$ up to equivalence. 
This is what we shall study next.

\subsection{Admissibility}

We consider an abstract version of $A_v$.
Let $G$ be a finite group.

\begin{defn}\label{defn:AID}
Let $D$ be a subgroup of $G$ and $I$ a normal subgroup of $D$ such that $D/I$ is a cyclic group.
For each generator $\varphi$ of $D/I$, we define a $\Z[G]$-module
\[
A_{I, \varphi} = \Z[G] \otimes_{\Z[D]} \Z[D/I](g_{\varphi}),
\quad
g_{\varphi} = 1 - \varphi^{-1} + \# I \in \Z[D/I].
\]
\end{defn}

We would like to define $A_{I, D}$ as $A_{I, \varphi}$ with
any choice of $\varphi$.
Of course this raises an immediate question.
Indeed it is fairly easy to check
on examples that the isomorphism class of 
$A_{I,\varphi}$ does depend in general on
the choice of the generator $\varphi$ of the cyclic group $D/I$.  Luckily, the
equivalence class of 
$A_{I,\varphi}$ is independent of that choice:

\begin{prop}
In Definition \ref{defn:AID}, the class of the module $A_{I, \varphi}$ 
modulo equivalence is independent of the choice of $\varphi$.
\end{prop}

This proposition allows us to write $A_{I, D}$ for the equivalence class of $A_{I, \varphi}$ in the category of finite $\Z[G]$-modules.

\begin{proof}
Essentially this
important fact is already shown as implication (iii) $\Rightarrow$ (i) of \cite[Proposition 5.8]{GK2}.  We do not need the full content of that proposition here; on the
other hand, it assumed $G$ to be abelian. Therefore we prove
here what we need; this is no more than a slight elaboration upon \cite[Remark 5.11]{GK2}.

Let
\[
\cL_{I, \varphi} := \left(\nu_I, 1 - \frac{\nu_I}{\# I} \varphi^{-1} \right)
\]
be the fractional ideal of $\Z[G]$ generated by the two elements.
Then as in \cite[Proposition 5.10]{GK2}, this lattice $\cL_{I, \varphi}$ represents $\Phi(\omega^{-1}(A_{I, \varphi}))$ up to projective equivalence, where $\omega^{-1}$ denotes the inverse of $\omega^1$.
The proof is valid without essential changes; we only note that the element $\nu_I$ is central in $\Z[D]$ as $I$ is a normal subgroup of $D$.
Then it is enough to show that the projective equivalence class of $\cL_{I, \varphi}$ is independent of the choice of $\varphi$.

We can also verify the exact sequence
\[
0 \to \Z[G/I] \overset{\nu_I}{\to} \cL_{I, \varphi} \to \Z[G]/(\nu_I) \to 0
\]
in \cite[Proposition 5.12]{GK2}.
We then compute its extension class: We have
\[
\Ext^1_{\Z[G]}(\Z[G]/(\nu_I), \Z[G/I]) \simeq \Z[G/I]/(\# I)
\]
and the extension class is the class of $\varphi^{-1} - 1$.
Then, even if we change $\varphi$, the extension class remains unchanged up to a unit.
Therefore, the isomorphism class of $\cL_{I, \varphi}$ is unchanged.
\end{proof}

Let us introduce a notion of local realizability for a pair $(D, I)$.

\begin{defn}\label{defn:rea_loc}
Let $\ell$ be a prime number.
For a finite group $D$ and a subgroup $I \subset D$, we say that the pair $(D, I)$ 
is {\it realizable in finite extensions of $\Q_{\ell}$} if there are 
\begin{itemize}
\item[(i)]
a finite extension $K_v/\Q_{\ell}$ and a finite Galois extension $\Lp_w/K_v$, and
\item[(ii)]
an isomorphism $\psi: \Gal(\Lp_w/K_v) \simeq D$ of groups
\end{itemize}
such that $\psi(I(\Lp_w/K_v)) = I$, where $I(\Lp_w/K_v)$ denotes the inertia group in $\Gal(\Lp_w/K_v)$.
In this case, we say that the pair $(D, I)$ is realizable by $\Lp_w/K_v$.
\end{defn}

\begin{lem}\label{lem:real_allow}
If $(D, I)$ is realizable in finite extensions of $\Q_{\ell}$, then the following hold.
   \begin{itemize}
     \item[(1)]  The subgroup $I$ of $D$ is normal and $D/I$ is cyclic.
     \item[(2)]  The $\ell$-Sylow subgroup $I_{(\ell)}$ of $I$ is normal in $I$ and
           $I/I_{(\ell)}$ is cyclic.
   \end{itemize}
\end{lem}

\begin{proof}
(1) This is well-known: $D/I$ is generated by the Frobenius.
(2) 
The subgroup $I_{(\ell)}$ of $I$ is the wild
ramification group; this is normal in $I$, again with cyclic factor group
(the ``tame part'' of ramification).
\end{proof}

Now we are finally ready to define admissible objects. 
Let $G = \Gamma \times \{1, j\}$ be as in Definition \ref{defn:rea}.
For subgroups $I \subset D \subset G$ such that $(D, I)$ is realizable in finite extensions of $\Q_{\ell}$ for some prime number $\ell$, by using Lemma \ref{lem:real_allow}(1), we have $A_{I, D}$, which is a finite $\Z[G]$-module 
well defined up to equivalence.

\begin{defn} \label{defn:adm}
We consider the monoid $\Cmodsim$ of finite $\Z'[\Gamma]$-modules up to equivalence.
We define $\cZ^{\adm} \subset \Cmodsim$ as the submonoid generated by $\omega^1(A_{I, D}^-)$, where $(D, I)$ runs through pairs of subgroups of $G$ that are realizable in finite extensions of $\Q_{\ell}$ for some prime number $\ell$.
An element of $\Cmodsim$ is said to be admissible if it is in $\cZ^{\adm}$.
In the same way, we define $\cZ_p^{\adm} \subset \Cpmodsim$, considering $\Z_p[\Gamma]$-modules instead of $\Z'[\Gamma]$-modules.
\end{defn}

Recall that for $L/K$ as in Definition \ref{defn:rea}, $\Cl_L^{T, -}$ is equivalent to $\bigoplus_{v \in \Sigma_f} \omega^1(A_{I_w, D_w}^-)$ (Theorem \ref{thm:Cl_class}) and each $(D_w, I_w)$ is clearly realizable by $L_w/K_v$.
This tells us:

\begin{prop}
We have $\cZ^{\rea} \subset \cZ^{\adm}$.
\end{prop}

Let us provide sufficient conditions for the triviality of $A_{I, D}^-$:

\begin{lem}[{Cf.~\cite[Lemma 2.7]{GK2}}]\label{lem:vanish}
Let $(D, I)$ be a pair of subgroups of $G$ as in Definition \ref{defn:AID}.
Let $p$ be an odd prime number and set ${}_p A_{I, D} = \Z_p \otimes_{\Z} A_{I, D}$.
\begin{itemize}
\item[(1)]
If $p \nmid \# I$, then ${}_p A_{I, D} \sim 0$ as a finite $\Z_p[G]$-module.
\item[(2)]
If $j \in D$, then ${}_p A_{I, D}^- \sim 0$ as a finite $\Z_p[\Gamma]$-module.
\end{itemize}
\end{lem}

\begin{proof}
(1)
Suppose $p \nmid \# I$.
Then $\Z_p[D/I]$ is a projective $\Z_p[D]$-module.
Since $g_{\varphi} = 1 - \varphi^{-1} + \# I$ is a non-zero-divisor of $\Z_p[D/I]$, it follows 
that $\Z_p[D/I]/(g_{\varphi})$ has $\pd \leq 1$ over $\Z_p[D]$.
Since $\Z_p[G]$ is  free over $\Z_p[D]$, we see that ${}_p A_{I, \varphi}$ has $\pd \leq 1$ over $\Z_p[G]$, as claimed.
 
(2)
By (1) we may assume that $p \mid \# I$.
Suppose moreover $j \in D$.
Take an integer $k \in \Z$ such that the class of $j$ is $\varphi^k$ in $D/I$.
Then we have
\begin{align*}
\left( \Z_p[D/I]/(g_{\varphi}, p) \right)^-
& = \Z_p[D/I]/(1 - \varphi^{-1} + \# I, p, j + 1)\\
& = \Z_p[D/I]/(1 - \varphi^{-1}, p, \varphi^k + 1)
= 0.
\end{align*}
Then Nakayama's lemma shows $\left( \Z_p[D/I]/(g_{\varphi}) \right)^- = 0$, so also ${}_p A_{I, \varphi}^- = 0$.
\end{proof}

\section{Admissible classes for metacyclic groups}\label{sec:metaab}

Let $p$ be a fixed odd prime number.
In this section, we assume that 
$G = \Gamma \times \{1,j\} $,
where $j$ will act as complex conjugation and $\Gamma$ is the metacyclic group
given as a semidirect product
$$  \Gamma = \langle \sigma, \tau : \sigma^p = 1 = \tau^r, 
         \tau\sigma\tau^{-1} = \sigma^s\rangle = C_p \rtimes C_r, $$
where the class of $s$ in $(\Z/p\Z)^{\times}$ has exact order $r$.
Consequently, $r$ is a divisor of $p-1$.

Let $\Cpmodsim$ be the monoid of finite $\Z_p[\Gamma]$-modules up to equivalence.
We have the submonoid $\cZ_p^{\adm} \subset \Cpmodsim$ of admissible classes as in Definition \ref{defn:adm}.
The aim of this section is to prove Theorem \ref{thm:11} on the monoid structures of $\cZ_p^{\adm}$ and $\Cpmodsim$.
In \S \ref{ss:classification}, we review the classification of $\Z_p[\Gamma]$-lattices.
Then in \S \ref{ss:Cpmodsim}, we determine the structure of $\Cpmodsim$, establishing Theorem \ref{thm:11}(1).
After a preliminary computation of lattices in \S \ref{ss:comp_latt}, we will determine the structure of $\cZ_p^{\adm}$ in \S \ref{ss:adm_latt}, establishing Theorem \ref{thm:11}(2).

\subsection{Classification theorem}\label{ss:classification}

As mentioned in Remark \ref{rem:Latpe_free},
 the structure of $\Lat_p^{\pe} = \Lat^{\pe}_{\Z_p[\Gamma]}$ is determined by 
the indecomposable lattices, whose classification is already known for our $\Gamma$.

A brief word on the history of the classification theorem: 
The first case to be covered was arbitrary odd $p$ and $r=2$, so
$\Gamm$ is dihedral, see \cite{Lee}. Just a little later, the case of
general $r$ was done in \cite{Pu}.  Our source of reference is
the very clear and concise
presentation in \cite[\S34E]{CR}.
Two quick remarks: That source
uses the letter $q$ instead of our $r$, and it makes the assumption
that $q$ is also prime, but this is never used. 

We now state the classification of lattices over $\Z_p[\Gamm]$.

\begin{thm}\label{thm:class_latt}
There are exactly $3r$ indecomposable $\Z_p[\Gamma]$-lattices modulo isomorphism.
We may call them
\[
L_1^i, L_2^i, L_3^i,
\quad
(i \in \Z/r\Z)
\]
so that the following hold.
\begin{itemize}
\item
We have $\rank_{\Z_p}(L_1^i) = 1$, $\rank_{\Z_p}(L_2^i) = p-1$, and $\rank_{\Z_p}(L_3^i) = p$.
\item
We have an exact sequence
\[
0 \to L_2^{i+1} \to L_3^i \to L_1^i \to 0.
\]
\item
We have an isomorphism $\bigoplus_{i \in \Z/r\Z} L_3^i \simeq \Z_p[\Gamma]$, so every $L_3^i$ is projective.
\item
The lattices $L_1^i$ and $L_2^i$ are not projective.
\end{itemize}
\end{thm}

Let us explain the construction of the lattices $L_k^i$ for $k = 1, 2, 3$.

Note that, since the abelianization of $\Gamma$ is $\Gamma/\lrang{\sigma} \simeq C_r$, there are $r$ group homomorphisms $\Gamma \to \Z_p^{\times}$.
More concretely, there is a unique root of unity $\rho \in \Z_p^{\times}$ that is congruent to $s$ modulo $p$, and we have a group homomorphism $\omega: \Gamma \to \Z_p^{\times}$ that sends $\tau$ to $\rho$.
Then the set of group homomorphisms $\Gamma \to \Z_p^{\times}$ is $\{\omega^i \mid i \in \Z/r\Z\}$.
We call $\omega$ the Teichm\"uller character.

$k=1$: Here $L_1^i$ is $\Z_p$, with trivial action of $\sigma$, and
$\tau$ acting via $\omega^i$. So all the $L_1^i$ are lattices already over
the quotient group $\Gamm/\langle \sigma \rangle$; indeed all indecomposable
$\Z_p$-lattices over this cyclic group of order $r$ are of rank one,
and given by one of the $L_1^i$. Another suggestive notation for
$L_1^i$ would be $\Z_p(i)$ (Tate twist). 

$k=3$: 
Recall that $\lrang{\tau}$ is a cyclic group of order $r$ and its characters are given by $\omega^i$ for $i \in \Z/r\Z$.
We then use the idempotents
\[
e_i = \frac{1}{r} \sum_{j=0}^{r-1} \omega^i(\tau^j) \tau^{-j} \in \Z_p[\lrang{\tau}] \subset \Z_p[\Gamma].
\]
The lattice $L_3^i$ is then defined to be the left ideal $\Z_p[\Gamma] e_i =
\Zp[\lrang{\sigma}]e_i$ in $\Zp[\Gamm]$.
Then $L_3^i$ is free of rank one over $\Z_p[\lrang{\sigma}]$, so it is indecomposable already over $\Z_p[\lrang{\sigma}]$, and the $\Z_p$-rank is $p$.

$k=2$: 
Let $\zeta$ be a primitive $p$-th root of unity.
We embed $\lrang{\tau} = C_r$ into the Galois group $\Gal(\Q_p(\zeta)/\Q_p) \simeq (\Z/p\Z)^{\times}$ by sending $\tau$ to $s$ modulo $p$.
Then $\Q_p(\zeta)$ can be regarded as a $\Q_p[\Gamma]$-module, with $\sigma$ acting
as multiplication by $\zeta$, and $\tau$ via Galois action.
For each $i \in \Z$, we have a $\Z_p[\Gamma]$-lattice $(1-\zeta)^i \Z_p[\zeta] \subset \Q_p(\zeta)$, and 
we call it $L_2^i$. 
Then $L_2^i$ is free of rank one over $\Z_p[\zeta]$, so it is indecomposable already over $\Z_p[\lrang{\sigma}]$, and the $\Z_p$-rank is $p-1$.

The isomorphism class of $L_2^i$  depends only on $i$ modulo $r$.
Some further explanations are perhaps helpful.
From the theory of Gauss sums one can deduce that there is an
element $g_i\in \Z_p[\zeta]$ with $(1-\zeta)$-adic valuation $i$ and
the property $\tau(g_i)= \omega^i(\tau)\cdot g_i$. Thus $L_2^i
= g_i \Z_p[\zeta]$, and we can interpret $L_2^i$ as $\Z_p(i) \otimes_{\Zp} \Z_p[\zeta]
=\Z_p[\zeta](i)$. 

There is the following link between the three series:
The lattice $L_3^0 = \Zp[\lrang{\sigma}] e_0$ contains the sublattice 
$(1-\sigma)\Zp[\lrang{\sigma}]e_0$, which is isomorphic
to $L_2^1=(1-\zeta)\Z_p[\zeta]$, by sending $\sigma$ to $\zeta$ and $e_0$ to 1.
So one gets a short exact sequence $0 \to L_2^1 \to L_3^0 \to \Zp(0)
\to 0$, which does not split, and by twisting for all $i$:
$$ 0 \to L_2^{i+1} \to L_3^i \to L_1^i= \Zp(i) \to 0. $$

This finishes the explanation of the construction of the indecomposable lattices.

\subsection{The structure of $\Cpmodsim$}\label{ss:Cpmodsim}

Now we are ready to prove Theorem \ref{thm:11}(1).
Recall that we have an injective homomorphism $\Phi: \Cpmodsim \hookrightarrow \Lat^{\pe}_p$.

\begin{thm}
The following hold.
\begin{itemize}
\item[(1)]
The monoid $\Lat^{\pe}_p$ is free of rank $2r$.
\item[(2)]
The monoid $\Cpmodsim$ is isomorphic to 
\[
\left\{ (a_1, \dots, a_r, b_1, \dots, b_r) \in \bN^{2r} \, \middle| \, \sum_{i=1}^r a_i = \sum_{i=1}^r b_i \right\}.
\]
In particular, the monoid $\Cpmodsim$ is not free unless $r = 1$, and it
 generates a free abelian group of rank $2r-1$.
\end{itemize}
 \end{thm}

\begin{proof}  
(1) follows immediately from the classification theorem (Theorem \ref{thm:class_latt}).
Indeed, the basis of $\Lat^{\pe}_p$ is given by the classes of $L_1^i, L_2^i$ for $1 \leq i \leq r$.

(2)
By (1), an element of $\Lat^{\pe}_p$ can be written as
\[
L = \bigoplus_{i=1}^{r} (L_1^i)^{a_i} \oplus \bigoplus_{i=1}^{r} (L_2^{i+1})^{b_i}
\]
for $(a_1, \dots, a_r, b_1, \dots, b_r) \in \N^{2r}$.
Let us show that $L$ is in $\Phi(\Cpmodsim)$ if and only if $\sum_{i=1}^r a_i = \sum_{i=1}^r b_i$.

For $k=1,2,3$ and $i=1,\ldots,r$  let $V_k^i=\mathbb Q_p \otimes_{\Zp} L_k^i$.
The representations $V_1^i$ are exactly the degree-one representations of 
$\Gamma$, and $V_3^i \cong  V_1^i \oplus V_2^{i+1}$ by the short exact sequence in Theorem \ref{thm:class_latt}. Now in contrast with the $L_2^i$, the $V_2^i$ do
not differ essentially from each other. To the contrary: as $L_2^{i+1}$ is
by construction a sublattice of index $p$ in $L_2^i$, all the $V_2^i$ simply agree,
and give the unique representation $V$ of degree $p -1$ of $\Gamma$. 

The projective lattices are exactly those that can be written as a direct sum
of copies of some $L_3^i$'s. 
Therefore, by Proposition \ref{prop:fin_lat}, the condition for $L$ to be in the image of $\Phi$ is equivalent to 
\[
\Q_p \otimes L \simeq \bigoplus_{i=1}^r (V_3^i)^{d_i}
\]
for some $d_i$.
This is rewritten as
\[
\bigoplus_{i=1}^{r} (V_1^i)^{a_i} \oplus \bigoplus_{i=1}^{r} (V_2^{i+1})^{b_i} 
\simeq \bigoplus_{i=1}^{r} (V_1^i)^{d_i} \oplus \bigoplus_{i=1}^{r} (V_2^{i+1})^{d_i}.
\]
By the remarks above, the condition for $d_i$ is now equivalent to $d_i = a_i$ for all $i$ and $\sum_{i=1}^r d_i = \sum_{i=1}^r b_i$.
This is equivalent to $\sum_{i=1}^r a_i = \sum_{i=1}^r b_i $.

Now the monoid $\Cpmodsim$  obviously generates a free subgroup of rank $2r-1$ in $\Z^{2r}$.
On the other hand, there are $r^2$ indecomposable elements in $\Cpmodsim$, so $r^2 > 2r - 1$ shows that it is not a free monoid unless $r = 1$.
\end{proof}

\subsection{Computation of the associated lattices}\label{ss:comp_latt}

The main objective of this subsection is to
determine the lattice $\Phi({}_p A_{I, D}^-)$ for various pairs $(I, D)$ with $I \subset D \subset G = \Gamma \times \{1, j\}$ as in Definition \ref{defn:AID}.

By Lemma \ref{lem:vanish}, we may assume that $p \mid \# I$, that is, $I$ contains $\lrang{\sigma}$.
By the same lemma, we may also assume that $j$ is not contained in $D$.
In this case, we can specify the possible pairs:

\begin{lem}\label{lem:cases}
Suppose $p \mid \# I$ and $j \not \in D$.
Then either
\begin{itemize}
\item[(I)]
$D = \lrang{\sigma, \tau^{r/e}}$ and $I = \lrang{\sigma, \tau^{r/d}}$ for some divisors $d \mid e \mid r$, or
\item[(II)]
$D = \lrang{\sigma, \tau^{r/e} j}$ and $I = \lrang{\sigma, (\tau^{r/e} j)^{e/d}}$ for some divisors $d \mid e \mid r$ such that $e$ is even.
\end{itemize}
\end{lem}

\begin{proof}
The natural map
\[
D/\lrang{\sigma} \hookrightarrow  G 
      /\lrang{\sigma} \simeq C_r \times \lrang{j} \twoheadrightarrow C_r
\]
is injective, so both $I/\lrang{\sigma}$ and $D/\lrang{\sigma}$ are cyclic groups.
Let $d$ and $e$ be the order of $I/\lrang{\sigma}$ and $D/\lrang{\sigma}$, respectively.
Then the above map sends a generator of $D/\lrang{\sigma}$ to $\tau^{r/e}$, so a generator of $D/\lrang{\sigma}$ is either $\tau^{r/e}$ or $\tau^{r/e} j$.
In the latter case, $e$ must be even,
since otherwise $j = (\tau^{r/e}j)^e$ would be in $D$, which was excluded.
\end{proof}

\begin{lem}
As $\Z_p[\Gamma]$-modules, we have 
\[
{}_p A_{I, D}^- \sim
\begin{cases}
\F_p[C_r]/(\tau^{r/e} - 1) \simeq \F_p[C_r/\lrang{\tau^{r/e}}]& \text{in case (I)}\\
\F_p[C_r]/(\tau^{r/e} + 1) & \text{in case (II)}.
\end{cases}
\]
\end{lem}

\begin{proof}
(I)
We may take $\tau^{r/e}$ as a generator of $D/I$, and then
\[
{}_p A_{I, \tau^{r/e}}^-
\simeq \Z_p[C_r \times \lrang{j}]/(\tau^{r/d} - 1, 1 - \tau^{-r/e} + pd, j + 1).
\]
As ideals we can compute
\begin{align*}
(\tau^{r/d} - 1, 1 - \tau^{-r/e} + pd)
& = ((1 + pd)^{e/d} - 1, 1 - \tau^{-r/e} + pd)\\
& = (p, 1 - \tau^{-r/e} + pd)\\
& = (p, 1 - \tau^{-r/e}).
\end{align*}
Therefore, we obtain
\begin{align*}
{}_p A_{I, \tau^{r/e}}^-
& \simeq \Z_p[C_r \times \lrang{j}]/(p, 1 - \tau^{-r/e}, j + 1)\\
& \simeq \F_p[C_r]/(1 - \tau^{-r/e}).
\end{align*}

(II)
In the same way, we have
\[
{}_p A_{I, \tau^{r/e} j}^-
\simeq \Z_p[C_r \times \lrang{j}]/((\tau^{r/e}j)^{e/d} - 1, 1 - \tau^{-r/e} j + pd, j + 1).
\]
As ideals of $\Z_p[C_r \times \lrang{j}]$ we can compute
\begin{align*}
& ((\tau^{r/e}j)^{e/d} - 1, 1 - \tau^{-r/e} j + pd, j + 1)\\
& \quad = ((1 + pd)^{e/d} - 1, 1 + \tau^{-r/e} + pd, j+1)\\
& \quad = (p, 1 + \tau^{-r/e} + pd, j+1)\\
& \quad = (p, 1 + \tau^{-r/e}, j+1).
\end{align*}
Therefore, we obtain
\begin{align*}
{}_p A_{I, \tau^{r/e} j}^-
& \simeq \Z_p[C_r \times \lrang{j}]/(p, 1 + \tau^{-r/e}, j+1)\\
& \simeq \F_p[C_r]/(1 + \tau^{-r/e}).
\end{align*}
\end{proof}

\begin{rem}\label{rem:Indep_I}
It is remarkable that ${}_p A_{I, D}^-$ is independent of $I$ in both cases (I) and (II).
This makes it easier to study the admissible classes.
\end{rem}

\begin{lem}
For each $e \mid r$, we have
\[
\Phi(\F_p[C_r]/(\tau^{r/e} - 1)) \sim \bigoplus_{e \mid i} (L_1^i \oplus L_2^{i+1})
\]
up to projective equivalence, where $1 \leq i \leq r$.
\end{lem}

\begin{proof}
First we show $\Phi(\F_p) = L_1^0 \oplus L_2^1$, which is the lemma for $e = r$.
Since $\Z_p[\lrang{\sigma}]$ is a projective $\Z_p[\Gamm]$-module, we can take a lattice $\Phi(\F_p)$ so that
\[
0 \to \Phi(\F_p) \to \Z_p[\lrang{\sigma}] \to \F_p \to 0
\]
is exact.
This shows that $\Phi(\F_p)$ is the class of the maximal ideal of $\Z_p[\lrang{\sigma}]$, which is then isomorphic to $L_1^0 \oplus L_2^1$.
For the general case, we only have to observe
\[
\F_p[C_r]/(\tau^{r/e} - 1) \simeq \bigoplus_{e \mid i} \F_p(i)
\]
and $\Phi(\F_p(i)) = \Phi(\F_p)(i)$, where $(i)$ denotes the Tate twists.
\end{proof}

Combining these lemmas, we obtain the goal in this subsection:

\begin{thm}\label{thm:Phi_A}
We have
\[
\Phi({}_p A_{I, D}^-)
\sim
\begin{cases}
0 & \text{unless $p \mid \# I$ and $j \not \in D$}\\
\bigoplus_{e \mid i} (L_1^i \oplus L_2^{i+1}) & \text{in case (I)}\\
\bigoplus_{\frac{e}{2} \mid i, e \nmid i} (L_1^i \oplus L_2^{i+1}) & \text{in case (II)}
\end{cases}
\]
up to projective equivalence, where $1 \leq i \leq r$.
\end{thm}

\begin{proof}
In Case (I), the claim is already proved.
For Case (II), we only have to use the isomorphism
\[
\F_p[C_r]/(\tau^{2r/e} - 1) \simeq \F_p[C_r]/(\tau^{r/e} - 1) \oplus \F_p[C_r]/(\tau^{r/e} + 1).
\]
\end{proof}

\subsection{The structure of $\cZ_p^{\adm}$}\label{ss:adm_latt}

We can now prove Theorem \ref{thm:11}(2).

\begin{thm}\label{thm:11(2)_pf}
The monoid $\cZ_p^{\adm}$ of admissible modules over $\Z_p[\Gamma]$ is a free monoid of rank $\sigma_0(r)$, 
where $\sigma_0(r)$ denotes the number of positive divisors of $r$.
\end{thm}

\begin{proof}
Recall that $\cZ_p^{\adm} \subset \Cpmodsim$ is the monoid generated by $\omega^1({}_p A_{I, D}^-)$ for the various $(D, I)$
that are realizable in finite extensions of $\Q_{\ell}$ for some prime $\ell$.
By using the automorphism induced by $\omega^1$ and $\Phi$, we see that $\cZ_p^{\adm}$ is isomorphic to the submonoid of $\Lat^{\pe}_p$ generated by $\Phi({}_p A_{I, D}^-)$.
Let $L(e) =  \bigoplus_{e \mid i} (L_1^i \oplus L_2^{i+1})$ for any $e \mid r$, and $L'(e) =
 \bigoplus_{\frac{e}{2} \mid i, e \nmid i} (L_1^i \oplus L_2^{i+1})$ for even $e \mid r$.
 Then by Theorem \ref{thm:Phi_A}, the monoid in question is contained in the monoid generated by these elements $L(e)$ and $L'(e)$.
It is not obvious that they coincide: we will need to establish local realizability of 
a suitably large supply of pairs $(D, I)$.
Nonetheless, this is indeed the case by Remark \ref{rem:Indep_I} and Proposition \ref{prop:real_minimal} established later (for now we do not need the extra properties (a)--(c), and the construction is quite simple).
As a result,
we only have to determine the submonoid of $\Lat_p^{\pe}$ generated by $L(e)$ and $L'(e)$.

We clearly have 
\[
L(e/2) = L(e) + L'(e)
\]
for even $e$ (as already used to determine $L'(e)$ in Theorem \ref{thm:Phi_A}).
Therefore, $L(e/2)$ for even $e$ can be removed from the set of generators, so the monoid is 
generated by $L(e)$ for $e$ with $2e \nmid r$ and $L'(e)$ for even $e$.
Now it is enough to show that these elements are linearly independent.
This is indeed quite easy: We only have to look at the coefficients of $L_1^1, \dots, L_1^r$.

To be precise, for each $e \mid r$, let $v(e) \in \Z^r$ be the row vector whose $i$-th component is $1$ if $e \mid i$ and $0$ otherwise.
For each even $e \mid r$, let $v'(e) \in \Z^r$ be the row vector whose $i$-th component if $1$ if $\frac{e}{2} \mid i, e \nmid i$ and $0$ otherwise.
Then our claim is that $v(e)$ for $e$ with $2e \nmid r$ and $v'(e)$ for even $e$ are linearly independent over $\Z$.
To show this, since $v(e/2) = v(e) + v'(e)$ for even $e$, we instead show that $v(e)$ for all $e \mid r$ are linearly independent.
This is easy: the $\sigma_0(r) \times r$ matrix whose rows are $v(e)$ is in row echelon form without zero rows,
 and therefore has maximal rank $\sigma_0(r)$. 
 To help visualizing this, we write out the matrix for $r = 6$:
\[
   \begin{pmatrix}
   v(1) \\ v(2) \\ v(3) \\ v(6)
   \end{pmatrix}
=
   \begin{pmatrix}
     1 & 1 & 1 & 1 & 1 & 1  \\
		 0 & 1 & 0 & 1 & 0 & 1  \\
		 0 & 0 & 1 & 0 & 0 & 1  \\
		 0 & 0 & 0 & 0 & 0 & 1  	
   \end{pmatrix}.
\]
\end{proof}
 
 \begin{rem}
 Let $\omega^{-1}$ be the inverse of $\omega^1: \Cpmodsim \overset{\simeq}{\to} \Cpmodsim$.
By the proof, we have shown that the basis of $\Phi(\omega^{-1}(\cZ_p^{\adm}))$ is $L(e)$ for $e$ with $2e \nmid r$ and $L'(e)$ for even $e$.
Using the Heller operator $\Omega$, we have $\Phi(\omega^{-1}(\cZ_p^{\adm})) = \Omega^{-1}(\Phi(\cZ_p^{\adm}))$.
Therefore, $\Phi(\cZ_p^{\adm})$ has a basis $\Omega(L(e))$ for $e$ with $2e \nmid r$ and $\Omega(L'(e))$ for even $e$.
We can show that $\Omega(L_1^i) = L_2^{i+1}$ and $\Omega(L_2^i) = L_1^i$ (we omit the details), which gives a complete description of $\Phi(\cZ_p^{\adm})$.
 \end{rem}
 

\section{Realizability problem}\label{sec:realize}

We recall (Definition \ref{defn:rea}) that a module class is called realizable if
it is the class of $\Cl^{T, -}_L$ or ${}_p \Cl^{T, -}_L$ for some CM extension
$L/K$ with Galois group $G = \Gamma \times \{1, j\}$, and that 
realizable classes are admissible.
We will prove a partial converse to this statement. This requires
a lot of preparation.
In this section, we focus on how to construct totally real extensions with Galois group $\Gamma$ and prescribed ramification styles.
This will be applied in the next section to the construction of CM extensions.

\subsection{Realizability}

Let us begin with a formulation of the problem.
Let $\Gamma$ be a finite group.
Let $I_1, \dots, I_n, D_1, \dots, D_n$ be a family of subgroups of 
$\Gamma$ such that $I_i \subset D_i$ ($1 \leq i \leq n$).
In what follows, $\Lp$ plays the role of $L^+$ in the previous sections.

\begin{defn}\label{defn:rea_gl}
We say that the family $\{(D_i, I_i)\}_{1 \leq i \leq n}$ of subgroups of $\Gamma$ is 
{\it realizable in finite extensions of $\Q$} if there are 
\begin{itemize}
\item[(i)]
a finite extension $K/\Q$ and a finite Galois extension $\Lp/K$ such that both $\Lp$ and $K$ are totally real,
\item[(ii)]
an isomorphism $\psi: \Gal(\Lp/K) \simeq \Gamma$ of groups, and 
\item[(iii)]
a family of finite primes $w_1, \dots, w_n$ of $\Lp$ lying above distinct primes of $K$
\end{itemize}
such that, for each $1 \leq i \leq n$, we have
\[
\psi(D_{w_i}(\Lp/K)) = D_i,
\quad
\psi(I_{w_i}(\Lp/K)) = I_i,
\]
where $I_{w_i}(\Lp/K)$ and $D_{w_i}(\Lp/K)$ denote the inertia and decomposition groups 
respectively of $w_i$ in $\Gal(\Lp/K)$.
In this case, we say that the family $\{(D_i, I_i)\}_{1 \leq i \leq n}$ is realizable by $\Lp/K$.
\end{defn}

In the previous paper \cite[Theorem 4.2]{GK2} of the same authors, we established
 the realizability for the case where $\Gamma$ is abelian and the order of $\Gamma$ is
 odd (under the necessary conditions explained in Remark \ref{rem:nec_loc} below).
The main purpose of this section is to discuss the nonabelian case, 
with a special focus on $C_p \rtimes C_r$.

\begin{rem}
In the formulation of realizability, we do not consider the ramification of primes other than 
(the conjugates of) $w_1, \dots, w_n$.
This will turn out to be harmless for our final applications (see the proof of Theorem \ref{thm:rea_adm}).
\end{rem}

It is clear that, if the family $\{(D_i, I_i)\}_{1 \leq i \leq n}$ is realizable in finite 
extensions of $\Q$, then each $(D_i, I_i)$ is realizable in finite extensions of 
$\Q_{\ell_i}$ (Definition \ref{defn:rea_loc}), where $\ell_i$ is the prime number lying below $w_i$.
The main results of this section give a partial converse to this observation.

\begin{rem}\label{rem:nec_loc}
Recall that Lemma \ref{lem:real_allow} provides necessary conditions for local realizability.
Moreover, for a finite extension $K_v/\Q_{\ell}$, the Galois group of the maximal tamely 
ramified extension of $K_v$ has a presentation of the form
\[
\lrang{\sigma, \tau :  \sigma \tau \sigma^{-1} = \tau^{N(v)}},
\]
where $N(v)$ denotes the order of the residue field of $K_v$ (see \cite[Theorem (7.5.3)]{NSW08}).
Therefore, the quotient group $D/I_{(\ell)}$ is a quotient of this group and $I/I_{(\ell)}$ is 
generated by the image of $\tau$.
\end{rem}

\subsection{Local extensions to global extensions: results}

For technical reasons, the main results of this section involve $\hat{\Z}$-extensions.
Let $\Q^{\cyc}$ be the unique $\hat{\Z}$-extension of $\Q$.
In other words, $\Q^{\cyc}$ is the composite of the cyclotomic $\Z_q$-extensions 
of $\Q$ for all prime numbers $q$. As the local counterpart, for each prime 
number $\ell$, let $\Q_{\ell}^{\cyc}$ be the composite of the cyclotomic $\Z_q$-extensions 
of $\Q_{\ell}$ for all prime numbers $q$.

We introduce the notion of realizability in finite extensions of $\Q^{\cyc}$ (resp.~of $\Q_{\ell}^{\cyc}$) 
in the same way as Definition \ref{defn:rea_gl} (resp.~Definition \ref{defn:rea_loc}):
We simply replace $\Q$ by $\Q^{\cyc}$ (resp.~$\Q_{\ell}$ by $\Q_{\ell}^{\cyc}$).

Note that the properties in Remark \ref{rem:nec_loc} are also necessary for the
 realizability in finite extensions of $\Q_{\ell}^{\cyc}$.
The realizability in finite extensions of $\Q_{\ell}^{\cyc}$ additionally requires that the index $(D: I)$ is an $\ell$-power,
since $\Q_{\ell}^{\cyc}$ already contains the unramified $\Z_q$-extensions of 
$\Q_{\ell}$ for any primes $q \neq \ell$.

Our main results try to connect four different but similar statements; the statements 
and their interplay can be roughly visualized as follows.

$$ \xymatrix{
	\text{global realizability/$\Q^{\cyc}$} 
		\ar@<0.5ex>@{=>}[r]^-{\text{clear}}
		\ar@<0.5ex>@{=>}[d]_-{\text{Proposition \ref{prop:going_down}(1)}}
	& \text{local realizability/$\Q_{\ell}^{\cyc}$} 
		\ar@<0.5ex>@{->}[l]^-{\text{Theorem \ref{thm:main_app}}} 
		\ar@<0.5ex>@{=>}[d]^-{\text{Proposition \ref{prop:going_down}(2)}}\\
	\text{global realizability/$\Q$}
		\ar@{=>}[r]_-{\text{clear}}
	& \text{local realizability/$\Q_{\ell}$}
		\ar@<0.5ex>@{->}[u]^-{\text{Proposition \ref{prop:going_up}}}
}  $$
The arrows ``$\Rightarrow$'' are unconditional, while the arrows ``$\rightarrow$'' require technical assumptions.

\subsection{Going down}

We have satisfactory ``going down'' properties for realizability.

\begin{prop}\label{prop:going_down}
The following hold.
\begin{itemize}
  \item[(1)] Suppose that a family $\{(D_i, I_i)\}_{1 \leq i \leq n}$ of subgroups 
of $\Gamma$ is realizable in finite extensions of $\Q^{\cyc}$.
Then the same family is realizable in finite extensions of $\Q$.
  \item[(2)] Suppose that $(D, I)$ is realizable in finite extensions of $\Q_{\ell}^{\cyc}$.
Then $(D, I)$ is realizable in finite extensions of $\Q_{\ell}$.
\end{itemize}
\end{prop}

\begin{proof}
We only show (1), because (2) can be shown in a similar (and easier) way.

By assumption, we have a finite Galois extension $\cLp/\cK$ of totally real finite extensions 
of $\Q^{\cyc}$, an isomorphism $\Gal(\cLp/\cK) \simeq \Gamma$, and finite primes 
$\sW_i$ of $\cLp$ ($1 \leq i \leq n$) whose decomposition and inertia groups correspond to $D_i$ and $I_i$.

It is easy to see that there is a finite Galois extension $\Lp/K$ of finite extensions of $\Q$ such that
$\cK = K \Q^{\cyc}$, $\cLp = \Lp \Q^{\cyc}$, and the restriction map gives an isomorphism
$$ \Gal(\cLp/\cK) \simeq \Gal(\Lp/K). $$
Indeed, we first take a finite extension $\Lp/K$ of finite extensions of $\Q$ such 
that $\cK = K \Q^{\cyc}$ and $\cLp = \Lp \Q^{\cyc}$.
We may assume that $\Lp/K$ is Galois by taking the Galois closure.
Then by replacing $K$ by $\Lp \cap \cK$, we obtain the claim.
Note that in this construction we may arbitrarily enlarge $K$ and $\Lp$.

Let $w_i$ be the prime of $\Lp$ lying below $\sW_i$.
By enlarging $\Lp$ and $K$ if necessary, we may assume that the primes of $K$ lying 
below $w_1, \dots, w_n$ are distinct. Moreover, by enlarging $\Lp$ and $K$ again, 
we see that the isomorphism $\Gal(\cLp/\cK) \simeq \Gal(\Lp/K)$ restricts to 
$D_{\sW_i}(\cLp/\cK) \simeq D_{w_i}(\Lp/K)$ and $I_{\sW_i}(\cLp/\cK) \simeq I_{w_i}(\Lp/K)$ for any $i$.
Now using  $\Lp/K$ we obtain the desired realizability.
\end{proof}

\subsection{Going up}

The following proposition requires restrictive hypotheses, but it will suffice for 
the particular case $\Gamma = C_p \rtimes C_r$, as will be discussed in \S \ref{sec:ex}.

\begin{prop}\label{prop:going_up}
Suppose that $(D, I)$ is realizable by $\Lp_w/K_v$ in finite extensions of $\Q_{\ell}$.
Further suppose the following.
\begin{itemize}
  \item[(a)] The index $(D: I)$ is an $\ell$-power.
  \item[(b)]  If $D \neq I$, then the cyclotomic $\Z_{\ell}$-extension of $K_v$ is totally ramified.
  \item[(c)] The index $(I: [D, D])$ is not divisible by $\ell$, where $[D, D]$ 
denotes the commutator subgroup of $D$.
\end{itemize}
Then $(D, I)$ is realizable by $\Lp_w^{\cyc}/K_v^{\cyc}$, where we set 
$\Lp_w^{\cyc} = \Lp_w \Q_{\ell}^{\cyc}$ and $K_v^{\cyc} = K_v \Q_{\ell}^{\cyc}$.
\end{prop}

\begin{proof}
We show that  restriction gives an isomorphism 
$$ \Gal(\Lp_w^{\cyc}/K_v^{\cyc}) \simeq \Gal(\Lp_w/K_v),  $$
and moreover this respects the inertia groups.
Let $M$ be the inertia field in $\Lp_w/K_v$.
Then the claim can be divided into two sub-claims (I) and (II):
\begin{itemize}
  \item[(I)] $M$ and $K_v^{\cyc}$ are linearly disjoint over $K_v$.
  \item[(II)]
$\Lp_w$ and $M^{\cyc}$ are linearly disjoint over $M$ and
$\Lp_w^{\cyc}/M^{\cyc}$ is totally ramified.
\end{itemize}

First we show (I), that is, $M \cap K_v^{\cyc} = K_v$.
If $D = I$, then $M = K_v$, so the claim is trivial.
By (a), $M/K_v$ is an $\ell$-extension, so $M \cap K_v^{\cyc}$ is contained 
in the cyclotomic $\Z_{\ell}$-extension of $K_v$.
Then condition (b) implies the claim.

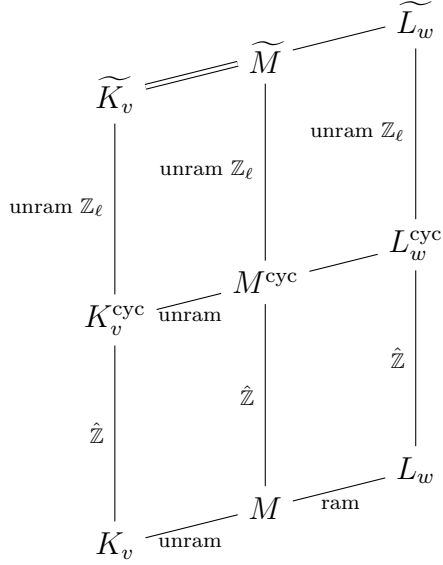
\begin{center}
\begin{tikzpicture}
	\node (F) at (0,0) {$K_v$};
	\node (M) at (2,0.5) {$M$};
	\node (K) at (4,1) {$\Lp_w$};
	\node (Fc) at (0,3) {$K_v^{\cyc}$};
	\node (Mc) at (2,3.5) {$M^{\cyc}$};
	\node (Kc) at (4,4) {$\Lp_w^{\cyc}$};
	\node (Ft) at (0,6) {$\wtil{K_v}$};
	\node (Mt) at (2,6.5) {$\wtil{M}$};
	\node (Kt) at (4,7) {$\wtil{\Lp_w}$};

	\draw (F) --(M);
	\draw (M) --(K);
	\draw (Fc) --(Mc);
	\draw (Mc) --(Kc);
	\draw[double, double distance=0.5mm] (Ft) --(Mt);
	\draw (Mt) --(Kt);
	\draw (F) --(Fc);
	\draw (Fc) --(Ft);
	\draw (M) --(Mc);
	\draw (Mc) --(Mt);
	\draw (K) --(Kc);
	\draw (Kc) --(Kt);
	
	\node[below] at (1, 0.25) {\tiny{unram}};
	\node[below] at (3, 0.75) {\tiny{ram}};
	\node[below] at (1, 3.25) {\tiny{unram}};

	\node[left] at (0, 1.5) {\tiny{$\hat{\Z}$}};
	\node[left] at (2, 2) {\tiny{$\hat{\Z}$}};
	\node[left] at (4, 2.5) {\tiny{$\hat{\Z}$}};

	\node[left] at (0, 4.5) {\tiny{unram $\Z_{\ell}$}};
	\node[left] at (2, 5) {\tiny{unram $\Z_{\ell}$}};
	\node[left] at (4, 5.5) {\tiny{unram $\Z_{\ell}$}};
\end{tikzpicture}
\end{center}

We now show (II). Let $\wtil{K_v}$, $\wtil{M}$, and $\wtil{\Lp_w}$ be the unramified 
$\Z_{\ell}$-extension of $K_v^{\cyc}$, $M^{\cyc} = M \Q_{\ell}^{\cyc}$, and $\Lp_w^{\cyc}$, 
respectively. Then (II) is equivalent to saying that $\Lp_w$ and $\wtil{M}$ are linearly disjoint 
over $M$, that is, $\Lp_w \cap \wtil{M} = M$.
We shall deduce this by showing that $\Lp_w \cap \wtil{M}/M$ is both unramified and totally ramified.
\begin{itemize}
   \item Since $\Lp_w/M$ is totally ramified, the subextension 
$\Lp_w \cap \wtil{M}/M$ is also totally ramified.
   \item Since $M/K_v$ is unramified, we have $\wtil{M} = \wtil{K_v}$, so $\wtil{M}/K_v$ is abelian.
In particular, the subextension $\Lp_w \cap \wtil{M}/K_v$ is also abelian.
This implies that $\Gal(\Lp_w \cap \wtil{M}/M)$ is a quotient of $I/[D, D]$, whose degree 
is prime to $\ell$ by condition (c). Since the degree of the inertia subgroup of $\Gal(\wtil{M}/M)$ 
is $\ell^{\infty}$, we conclude that $\Lp_w \cap \wtil{M}/M$ is unramified.
\end{itemize}
This completes the proof.
\end{proof}

\subsection{From local to global}

A finite group $\Gamma$ (whose unit element is denoted by $e$) is said to be {\it supersolvable} if there is a sequence
$$  \{ e \} = H_0 \subset H_1 \subset \dots \subset H_r = \Gamma,  $$
where each $H_i$ is a normal subgroup of $\Gamma$ and the factors $H_i/H_{i-1}$ 
are cyclic for $i = 1, \dots, r$. For instance, the metacyclic group $\Gamma = C_p \rtimes C_r$ is supersolvable, because we have a chain $\{e\} \subset C_p \subset \Gamma$.
On the other hand, the alternating group $A_4$ is not supersolvable, 
since $A_4$ does not have a nontrivial cyclic normal subgroup.
It is easy to see that all nilpotent groups are supersolvable and, in turn, 
all supersolvable groups are solvable.

For subgroups $D_1, \dots, D_n$ of a finite group $\Gamma$, we say that arbitrary 
conjugates of $D_1, \dots, D_n$ generate $\Gamma$ if, for any given 
$\gamma_1, \dots, \gamma_n \in \Gamma$, the subgroups 
$\gamma_1 D_1 \gamma_1^{-1}, \dots, \gamma_n D_n \gamma_n^{-1}$ generate $\Gamma$.

\begin{thm}\label{thm:main_app}
A family $\{(D_i, I_i)\}_{1 \leq i \leq n}$ of subgroups of $\Gamma$ is realizable 
in finite extensions of $\Q^{\cyc}$ if the following all hold.
\begin{itemize}
   \item  For each $1 \leq i \leq n$, the pair $(D_i, I_i)$ is realizable
 in finite extensions of $\Q_{\ell_i}^{\cyc}$ for some prime number $\ell_i$.
   \item $\Gamma$ is a supersolvable group.
\item Arbitrary conjugates of $D_1, \dots, D_n$ generate $\Gamma$.
\end{itemize}
\end{thm} 

The rest of this section is devoted to the proof of this theorem.
Indeed, the theorem will follow directly by combining Propositions \ref{prop:Kras}, 
\ref{prop:lift}, \ref{prop:surj}, and \ref{prop:cd_1}. 

A first step toward showing Theorem \ref{thm:main_app} is the following.

\begin{prop}\label{prop:Kras}
Given a family of finite extensions $K_{v_i}/\Q_{\ell_i}$ (all $\ell_i$ are prime numbers) 
for $1 \leq i \leq n$, there is a totally real finite extension $K/\Q$ and distinct finite 
primes $v_1, \dots, v_n$ of $K$ with the prescribed completions. --- 
The corresponding statement also holds over $\Q^{\cyc}$ instead of $\Q$.
\end{prop}

\begin{proof}
First we take a totally real number field $E$ that has distinct $\ell_i$-adic primes $u_i$ 
of degree $1$, i.e., $E_{u_i} \simeq \Q_{\ell_i}$.
We regard $K_{v_i}$ as a finite extension of $E_{u_i}$.
Then $K$ can be constructed as an extension of $E$ by using Krasner's lemma 
and the approximation theorem.

The final claim for $\Q^{\cyc}$ can be deduced from the claim for $\Q$.
\end{proof}

To prove Theorem \ref{thm:main_app}, it is convenient to use nonabelian Galois cohomology.
A good reference is Serre \cite[Chapter I, \S 5]{Ser02}.
In general, nonabelian Galois cohomology sets $H^i(\cG, \Gamma)$ involve a 
so-called coefficient group $\Gamma$ which need no longer be abelian.
Unless $\Gamma$ is abelian, we will only use the degree $i = 1$ and the case 
where the action of $\cG$ on $\Gamma$ is trivial.
The concrete definition in this case is as follows.

\begin{defn}
Let $\cG$ be any profinite group. Let $\Gamma$ be a (not necessarily abelian, 
discrete) group on which $\cG$ acts trivially. Then we put
\[
H^1(\cG, \Gamma) = \Hom(\cG, \Gamma)/\sim,
\]
where $\Hom$ denotes the set of continuous group homomorphisms 
and $f \sim f'$ if there is $\gamma \in \Gamma$ such that
\[
f(x) = \gamma^{-1} f'(x) \gamma,
\quad
x \in \cG.
\]
Note that $H^1(\cG, \Gamma)$ is not a group but only a pointed set; the distinguished 
point is the class of the constant homomorphism whose value is the unit element of $\Gamma$.
\end{defn}

\begin{prop}\label{prop:lift}
Let $K/\Q$ be a totally real finite extension and $v_1, \dots, v_n$ be distinct finite primes of $K$.
A family $\{(D_i, I_i)\}_{1 \leq i \leq n}$ of subgroups of $\Gamma$ is realizable by 
an extension of $K$ if the following hold.
\begin{itemize}
   \item For each $1 \leq i \leq n$, the pair $(D_i, I_i)$ is realizable 
by an extension of the completion $K_{v_i}$.
   \item The restriction map
$$ H^1(\cG_K, \Gamma) \to \prod_{i=1}^n H^1(\cG_{K_{v_i}}, \Gamma)  $$
is surjective, where $\cG_K$ (resp.~$\cG_{K_{v_i}}$) denotes the Galois group of the 
maximal totally real extension of $K$ (resp.~the absolute Galois group of $K_{v_i}$).
Here, we implicitly fix an embedding of the global field into the local field, which 
defines an injective homomorphism $\cG_{K_{v_i}} \hookrightarrow \cG_K$.
   \item  Arbitrary conjugates of $D_1, \dots, D_n$ generate $\Gamma$.
\end{itemize}
The corresponding statement also holds over $\Q^{\cyc}$ instead of $\Q$.
\end{prop}

\begin{proof}
We only show the claim for $\Q$, because the claim for $\Q^{\cyc}$ 
can be shown in the same way.

By the first assumption, the pair $(D_i, I_i)$ is realizable by an extension $\Lp_{w_i}/K_{v_i}$.
Let
$$ \psi_i: \cG_{K_{v_i}} \twoheadrightarrow \Gal(\Lp_{w_i}/K_{v_i}) \simeq D_i \hookrightarrow \Gamma  $$
be the induced homomorphism.

By the second assumption, there is a homomorphism $\psi: \cG_K \to \Gamma$ such 
that $\psi|_{\cG_{K_{v_i}}} \sim \psi_i$ in $H^1(\cG_{K_{v_i}}, \Gamma)$.
This means that there exists an element $\gamma_i \in \Gamma$ such that
$$ \psi(x) = \gamma_i \psi_i(x) \gamma_i^{-1}  $$
for any $x \in \cG_{K_{v_i}}$.
In particular, the image of $\psi$ contains $\gamma_i D_i \gamma_i^{-1}$.
By using the third assumption, this implies that $\psi$ is surjective.
Therefore, declaring $\Lp$ to be the fixed field of the kernel of $\psi$, 
we obtain an isomorphism $\Gal(\Lp/K) \simeq \Gamma$, which is also denoted by $\psi$.

Let $w_i'$ be the prime of $\Lp$ that corresponds to the fixed embedding defining 
$\cG_{K_{v_i}} \hookrightarrow \cG_K$. We write $D_{w_i'} = D_{w_i'}(\Lp/K)$ for the 
decomposition subgroup in $\Gal(\Lp/K)$.
Then 
$$ D_i = \psi_i(\cG_{K_{v_i}}) 
= \gamma_i^{-1} \psi(D_{w_i'}) \gamma_i
= \psi(\sigma_i^{-1} D_{w_i'} \sigma_i),  $$
where $\sigma_i \in \Gal(\Lp/K)$ is the element such that $\psi(\sigma_i) = \gamma_i$.
Then setting $w_i = \sigma_i^{-1}(w_i')$, we obtain $D_i = \psi(D_{w_i})$.
In the same way, we also have $I_i = \psi(I_{w_i})$.
\end{proof}

To apply Proposition \ref{prop:lift}, we have to establish the surjectivity of the restriction 
map between the global and local $H^1$ terms. Note that when $\Gamma$ is abelian
and the order of $\Gamma$ is not divisible by $4$, then the surjectivity is a special case 
of the Grunwald--Wang theorem (\cite[Theorem (9.2.3)(ii)]{NSW08}).
Our result in the nonabelian cases will be Proposition \ref{prop:surj} below.
The proof uses a kind of Grunwald--Wang theorem in induction steps, which is stated 
as Lemma \ref{lem:surj_H1} in advance.

For more generality, in what follows, the global fields are no longer assumed to be totally real but we impose an allowability in the following sense.

\begin{defn}\label{defn:allow_ext}
For a number field $K$, let $K_f/K$ be the maximal algebraic extension that is unramified 
at all infinite places and set $\cG_K = \Gal(K_f/K)$ (the subscript $f$ is put to mean 
``ramifications only at finite primes''). We introduce an ad hoc notion: 
An algebraic extension $\cK$ of $\Q$ is said to be {\it allowable} if $\cK$ is an 
intermediate field of $K_f/K$ for some number field $K$.
For an allowable algebraic extension $\cK/\Q$, we also define $\cG_{\cK} = \Gal(\cK_f/\cK)$ 
in the same way. Any totally real field $\cK$ is allowable and then $\cG_{\cK}$ is the same 
as the one in Proposition \ref{prop:lift}.
\end{defn}

\begin{lem}\label{lem:surj_H1}
Let $\cK$ be an allowable algebraic extension of $\Q$.
Let $\sV_1, \dots, \sV_n$ be distinct finite primes of $\cK$.
Let $A$ be a $\cG_{\cK}$-module whose order is a prime number.
Then the restriction map
$$ H^1(\cG_{\cK}, A) \to \prod_{i=1}^n H^1(\cG_{\cK_{\sV_i}}, A)   $$
is surjective.
\end{lem}

\begin{proof}
First suppose that $\cK$ is a number field and the action of $\cG_{\cK}$ on $A$ is trivial.
To clarify the situation, write $K = \cK$ and $v_i = \sV_i$.
Then by \cite[Theorem (9.2.3)(ii)]{NSW08}, the restriction map
\[
H^1(K, A) \to \prod_{i=1}^n H^1(K_{v_i}, A) \times \prod_{\text{$v$: real}} H^1(K_v, A)
\]
is surjective, where $v$ runs over the real places of $K$ and we are considering the cohomology of the absolute Galois groups (we are not in 
``the special case'' since the order of $A$ is not divisible by $4$).
Since the kernel of $H^1(K, A) \to \prod_{\text{$v$: real}} H^1(K_v, A)$ is 
equal to the image of the inflation map from $H^1(\cG_K, A)$, we obtain the claim.

By using the assumption that $\cK$ is allowable, taking the direct limit shows 
that the lemma is valid as long as the action of $\cG_{\cK}$ on $A$ is trivial.
Let us show that the general case can be reduced to this case.

Let $\cK(A)/\cK$ be the extension cut out by $A$.
By the assumption that $A$ is isomorphic to $\Z/q\Z$ as an abstract group for
 a prime number $q$, the Galois group $\Gal(\cK(A)/\cK)$ is isomorphic to a subgroup of $(\Z/q\Z)^{\times}$.
In particular, the order of $\Gal(\cK(A)/\cK)$ is prime to the order of $A$.
Note that $\cK(A)$ is also allowable since $\cK(A)/\cK$ is a finite extension.
Then by the established case, the restriction map
$$ H^1(\cG_{\cK(A)}, A) \to \prod_{i=1}^n \prod_{\sV' \mid \sV_i} H^1(\cG_{\cK(A)_{\sV'}}, A)  $$
is surjective, where $\sV'$ runs over primes of $\cK(A)$.
By taking $\Gal(\cK(A)/\cK)$-invariants, we obtain the lemma.
\end{proof}

\begin{prop}\label{prop:surj}
Let $\cK$ be an allowable algebraic extension of $\Q$.
Suppose that the cohomological dimension of $\cG_{\cK}$ satisfies $\cd \cG_{\cK} \leq 1$.
Let $\sV_1, \dots, \sV_n$ be distinct primes of $\cK$.
Then the restriction map
$$ H^1(\cG_{\cK}, \Gamma) \to \prod_{i=1}^n H^1(\cG_{\cK_{\sV_i}}, \Gamma)  $$
is surjective for any finite supersolvable group $\Gamma$.
\end{prop}

In \S \ref{ss:p-cd}, we will show that $\cd \cG_{\cK} \leq 1$ if $\cK$ contains $\Q^{\cyc}$.
This is the point where we need to work over $\Q^{\cyc}$ rather than over $\Q$.

\begin{proof}
We use induction on the order of $\Gamma$.
When $\Gamma$ is the trivial group, the proposition is clear.
Suppose that $\Gamma$ is nontrivial.
Then by the supersolvability, $\Gamma$ has a nontrivial, cyclic, normal subgroup $H$, 
and the quotient $\Gamma/H$ remains supersolvable. We may take $H$ to be simple, i.e., 
the order of $H$ is a prime number (since such a subgroup of a cyclic group is a characteristic subgroup).

Consider the diagram
\[
\xymatrix{
	H^1(\cG_\cK, \Gamma) \ar[r] \ar[d]_{\text{surj}}
	&  \prod_{i=1}^n H^1(\cG_{\cK_{\sV_i}}, \Gamma) \ar[d]\\
	H^1(\cG_\cK, \Gamma/H) \ar_-{\text{surj}}[r]
	&  \prod_{i=1}^n H^1(\cG_{\cK_{\sV_i}}, \Gamma/H).
}
\]
The lower horizontal arrow is surjective by the induction hypothesis.
The left vertical arrow is surjective by applying 
\cite[page 58, Proposition 46]{Ser02}, taking $\cd \cG_{\cK} \leq 1$ into account.

Let us show that the upper horizontal arrow is surjective.
Let $(\eta_i)_i \in \prod_{i=1}^n H^1(\cG_{\cK_{\sV_i}}, \Gamma)$ be any element.
Then we find an element $\xi \in H^1(\cG_\cK, \Gamma)$ whose image in 
$\prod_{i=1}^n H^1(\cG_{\cK_{\sV_i}}, \Gamma/H)$ is the same as that of $(\eta_i)_i$.
We write $\xi_i \in H^1(\cG_{\cK_{\sV_i}}, \Gamma)$ for the image of $\xi$, so $\xi_i$ 
and $\eta_i$ goes to the same element in $H^1(\cG_{\cK_{\sV_i}}, \Gamma/H)$.

We now use \cite[page 52, Corollary 2 to Proposition 39]{Ser02} that describes the 
fibers of $H^1(\cG_{\cK_{\sV_i}}, \Gamma) \to H^1(\cG_{\cK_{\sV_i}}, \Gamma/H)$.
Let $x$ be a group homomorphism $\cG_{\cK} \to \Gamma$ that represents $\xi$ 
and let ${}_x H$ be the twist of $H$ by $x$, which is a $\cG_{\cK}$-module.
Then $\xi_i$ and $\eta_i$ are in the same orbit by the action of $H^1(\cG_{\cK_{\sV_i}}, {}_x H)$, 
that is, we have $\eta_i = \omega_i \xi_i$ for some $\omega_i \in H^1(\cG_{\cK_{\sV_i}}, {}_x H)$.

By Lemma \ref{lem:surj_H1} applied to $A = {}_x H$, we can find an element 
$\omega$ of $H^1(\cG_\cK, {}_x H)$ that goes to $(\omega_i)_i$.
Then the element $\omega \xi \in H^1(\cG_{\cK}, \Gamma)$ goes 
to $(\omega_i \xi_i)_i = (\eta_i)_i$, as desired.
\end{proof}

\section{Realizable classes for metacyclic groups}\label{sec:rea_metacyc}

The goal of this section is to prove Theorem \ref{thm:12}.
We consider $\Gamma = C_p \rtimes C_r$ with $C_p = \lrang{\sigma}$ and $C_r = \lrang{\tau}$ as in \S \ref{sec:metaab}.

\subsection{Construction of totally real extensions}\label{sec:ex}

Our first step is to apply the results in \S \ref{sec:realize} to derive global realizability (by totally real extensions) for family of pairs of subgroups of $\Gamma$.
Unfortunately, we cannot apply our method in general (see Remark \ref{rem:real_fail}),
but we can at least construct a good supply of extensions, providing the ones that will be necessary in the final realizability proof.

\begin{prop}\label{prop:real_minimal}
Let $(D, I)$ be a pair of subgroups of $\Gamma = C_p \rtimes C_r$ such that $D \supset C_p$ and $I = D$.
Then $(D, I)$ is realizable by $L_w/K_v$ in finite extensions of $\Q_{\ell}$ for some prime number $\ell$, satisfying the conditions (a), (b), and (c) in Proposition \ref{prop:going_up}.
\end{prop}

\begin{proof}
By assumption, we have $D = I = \lrang{\sigma, \tau^{r/e}}$ for some $e \mid r$ (see Lemma \ref{lem:cases}).
Note that conditions (a) and (b) are trivial as $D = I$.

Suppose $e = 1$, that is, $D = I = C_p$.
Then $(I: [D, D]) = p$, so any prime $\ell \neq p$ satisfies condition (c).
It is then enough to take $K_v$ to be an extension of $\Q_{\ell}$ large enough so that there is a totally ramified abelian extension $\Lp_w/K_v$ of degree $p$. 

Suppose $e > 1$, that is, $D = I \supsetneqq C_p$.
Then $(I: [D, D]) = (I: C_p) = e$, so $\ell = p$ satisfies condition (c).
We use the totally ramified $C_p \rtimes C_{p-1}$-extension $\Q_p(\mu_p, p^{1/p})/\Q_p$.
We only have to take $L_w = \Q_p(\mu_p, p^{1/p})$ and $K_v$ to be the intermediate field of $\Q_p(\mu_p)/\Q_p$ such that $[\Q_p(\mu_p): K_v] = e$.
\end{proof}

\begin{rem}\label{rem:real_fail}
More generally, let us discuss the case where $D \supset I \supset C_p$ and $D \supsetneqq C_p$ (not necessarily $D = I$).
Then we have $D = \lrang{\sigma, \tau^{r/e}}$ and $I = \lrang{\sigma, \tau^{r/d}}$ for some $d \mid e \mid r$ with $e > 1$.
In this case, condition (a) says that $e/d$ is an $\ell$-power, while condition (c) says that $d$ is not divisible by $\ell$.
Such an $\ell$ does not exist in general, so we cannot apply Proposition \ref{prop:going_up} to this case.
\end{rem}

We now obtain a very  important consequence.

\begin{thm}\label{thm:real_metab}
Let $\{(D_i, I_i)\}_{1 \leq i \leq n}$ be any
family of pairs of subgroups of $\Gamma = C_p \rtimes C_r$ such that $D_i \supset C_p$ and $I_i = D_i$.
Then the given family is realizable in finite extensions of $\Q$.
\end{thm}

\begin{proof}
By Proposition \ref{prop:real_minimal}, each $(D_i, I_i)$ is realizable in finite extensions of $\Q_{\ell_i}$ 
for some prime number $\ell_i$, and moreover we may apply Proposition \ref{prop:going_up} 
to see that $(D_i, I_i)$ is realizable in finite extensions of $\Q_{\ell_i}^{\cyc}$.

We apply Theorem \ref{thm:main_app}.
The group $\Gamma$ is supersolvable as already mentioned.
The final condition in Theorem \ref{thm:main_app} can be satisfied by adding 
$(\Gamma, \Gamma)$ to the family $\{(D_i, I_i)\}_{1 \leq i \leq n}$ if necessary.
As a result, the family $\{(D_i, I_i)\}_{1 \leq i \leq n}$ is realizable in finite extensions of $\Q^{\cyc}$.

Finally, Proposition \ref{prop:going_down}(1) shows the realizability in finite extensions of $\Q$.
\end{proof}

\subsection{Construction of CM extensions}

We close this section and the entire paper by using the last result to give a (partial) affirmative
answer to the question, which asked whether every admissible finite module
is realizable via some class group.
That is, we now prove Theorem \ref{thm:12}:

\begin{thm}\label{thm:rea_adm}
For the metacyclic group  $\Gamma = C_p \rtimes C_r$,
we have $\cZ_p^{\rea} = \cZ_p^{\adm}$.
\end{thm}

\begin{proof}
By Theorems \ref{thm:Cl_class} and \ref{thm:Phi_A}, 
 it is enough to show the following:
Let $D_1, \dots, D_n$ be any family of subgroups of $G$ such that $p \mid \# D_i$ and $j \not \in D_i$.
Then there is an extension $L/K$ such that the ramified primes $v$ with $p \mid \# I_v(L/K)$ 
and $j \not \in D_v(L/K)$ can be labeled $v_1, \dots, v_n$ so that $D_{v_i}(L/K) = D_i$.
Here,  the big point is that we are free to choose the inertia group as
any subgroup of the given decomposition group,  
because ${}_p A_{I, D}^-$ depends only on $D$, see Remark \ref{rem:Indep_I}.

To show this, we take $I_i = D_i$ and apply Theorem \ref{thm:real_metab} to $\{(\ol{D_i}, \ol{I_i})\}_{1 \leq i \leq n}$, 
where $\ol{(-)}$ denotes the image from $G = \Gamma \times \{1, j \}$ to $\Gamma$.
Note that $j \not \in D_i$ implies that $D_i \simeq \ol{D_i}$ and $I_i \simeq \ol{I_i}$.
As a consequence, we obtain a $\Gamma$-extension of totally real fields, denoted by $L^+/K$, that realizes this family.
Let $v_1, \dots, v_n$ be the places of $K$.

It remains to take an appropriate imaginary quadratic extension $F$ of $K$, which is combined with $L^+$ to give the desired $L$.
The required conditions for the quadratic extension $F/K$ are:
\begin{itemize}
\item
Each real place ramifies in $F/K$.
\item
For each $1 \leq i \leq n$, $F_{v_i}$ is isomorphic to the completion of $(L^+)^{D_i \cap \Gamma}$ at the prime above $v_i$.
This means:
\begin{itemize}
\item
If $D_i \subset \Gamma$ (so $D_i \cap \Gamma = D_i$), then $v_i$ splits in $F/K$, so we take $F_{v_i} \simeq K_{v_i}$.
\item
If $D_i \not \subset \Gamma$ (so $[D_i : D_i \cap \Gamma] = 2$), then the $v_i$-adic completion of $(L^+)^{D_i \cap \Gamma}$ is a quadratic extension of $K_{v_i}$, which we take as $F_{v_i}$.
\end{itemize}
\item
For each other unnecessary prime $v$ of $K$ such that $p \mid \# I_v(L^+/K)$, the completion $F_v$ is a quadratic extension of $K_v$ that is linearly disjoint from the completion of $(L^+)^{D_v(L^+/K)}$.
This is possible because $(L^+)^{D_v(L^+/K)}$ is a cyclic extension of $K$.
\end{itemize}
These are conditions on a  finite number of local extensions, so such an extension $F/K$ indeed exists.
By the second bullet point, we have $D_{v_i}(L/K) = D_i$.
By the final bullet point, the unnecessary $v$'s do not split in $L/L^+$, i.e., satisfy $j \in D_v(L/K)$.
\end{proof}

\appendix

    \section{The cohomological dimension}\label{ss:p-cd}

The purpose of this appendix is to prove Proposition \ref{prop:cd_1} on 
the cohomological dimension of $\cG_{\cK}$. Before that, we recall a known result 
on the cohomological dimension of the absolute Galois group that is denoted by $\ol{\cG}_{\cK} = \Gal(\ol{\cK}/\cK)$.
Let $p$ be a prime number.

\begin{prop}[{\cite[Corollary (8.1.18)]{NSW08}}]\label{prop:cd_2}
Let $\cK$ be an algebraic extension of $\Q$.
Suppose that $p^{\infty}$ divides the local degree of $\cK/\Q$ at any finite prime.
Suppose $\cK$ is totally imaginary if $p = 2$.
Then we have $\cd_p \ol{\cG}_{\cK} \leq 1$.
\end{prop}

\begin{proof}
Let us briefly review the proof.
Recall two basic properties of the $p$-cohomological dimension for a general profinite group $\cG$:
\begin{itemize}
   \item For a $p$-Sylow subgroup $\cG_{(p)}$ of $\cG$, we have $\cd_p \cG = \cd_p \cG_{(p)}$.
(See \cite[Corollary (3.3.6)]{NSW08}.)
   \item  If $\cG$ is a pro-$p$ group, then $\cd_p \cG \leq n$ if and only if $H^{n+1}(\cG, \Z/p\Z) = 0$.
(See \cite[Proposition (3.3.2)]{NSW08}.)
\end{itemize}
By using the first item, we may assume that $\ol{\cG}_{\cK}$ is a pro-$p$ group by 
considering the fixed field of a $p$-Sylow subgroup of $\ol{\cG}_{\cK}$.
Then by using the second item, we only have to show that $H^2(\ol{\cG}_{\cK}, \Z/p\Z) = 0$.

As $\ol{\cG}_{\cK}$ is a pro-$p$ group, $\cK$ contains the group $\mu_p$ of $p$-th 
roots of unity, so we have $\Z/p\Z \simeq \mu_p$ as $\ol{\cG}_{\cK}$-modules.
We use the exact sequence
$$ 0 \to \mu_p \to \ol{\cK}^{\times} \overset{p}{\to} \ol{\cK}^{\times} \to 0,  $$
where the map labeled ``$p$'' is $x \mapsto x^p$.
A key observation is that we have $H^2(\ol{\cG}_{\cK}, \ol{\cK}^{\times})[p^{\infty}] = 0$.
This follows from an investigation of the Brauer groups, using the assumption on the 
local degree of $\cK$ and that $\cK$ is totally imaginary if $p = 2$ (we omit the details here).
This, together with Hilbert 90, implies $H^2(\ol{\cG}_{\cK}, \mu_p) = 0$, as desired.
\end{proof}

Now let us establish what we actually need.
See Definition \ref{defn:allow_ext} for the definition of allowability and $\cG_{\cK} = \Gal(\cK_f/\cK)$.

\begin{prop}\label{prop:cd_1}
Let $\cK$ be an allowable algebraic extension of $\Q$.
Suppose that $p^{\infty}$ divides the local degree of $\cK/\Q$ at any finite prime.
Then we have $\cd_p \cG_{\cK} \leq 1$.
\end{prop}

\begin{proof}
As in the first paragraph of the proof of Proposition \ref{prop:cd_2}, we only have to 
show $H^2(\cG_{\cK}, \Z/p\Z) = 0$, assuming that $\cG_{\cK}$ is a pro-$p$ group.

Suppose $p \geq 3$. Then any $p$-extension is unramified at all infinite places.
Therefore, $\cG_{\cK}$ is the same as the maximal pro-$p$ quotient of $\ol{\cG}_{\cK}$.
By applying \cite[Corollary (10.4.8)]{NSW08} to the subfields of $\cK$ and taking 
the direct limit, we obtain an isomorphism
$$  H^2(\cG_{\cK}, \Z/p\Z) \simeq H^2(\ol{\cG}_{\cK}, \Z/p\Z).  $$
By Proposition \ref{prop:cd_2}, the right hand side vanishes.
This establishes the proposition for $p \geq 3$.

Now suppose $p = 2$. We modify the proof of Proposition \ref{prop:cd_2}.
We make use of exact sequences
\begin{equation}\label{eq:1}
0 \to \mu_2 \to \cK_f^{\times} \overset{2}{\to} (\cK_f^{\times})^2 \to 0
\end{equation}
and
\begin{equation}\label{eq:2}
0 \to (\cK_f^{\times})^2 \to \cK_f^{\times} \to \cK_f^{\times}/(\cK_f^{\times})^2 \to 0,
\end{equation}
where $\mu_2 = \{\pm 1\} \simeq \Z/2\Z$.

By \eqref{eq:1}, we shall obtain $H^2(\cG_{\cK}, \mu_2) = 0$ from 
Claims \ref{claim:11} and \ref{claim:12} below.

\begin{claim}\label{claim:11}
We have $H^2(\cG_{\cK}, \cK_f^{\times})[2^{\infty}] = 0$.
\end{claim}

\begin{proof}
It is a fundamental fact from class field theory that, for a finite extension $K'/K$ of 
number fields, we have a commutative diagram of exact sequences
\[
\xymatrix{
	0 \ar[r]
	& H^2(K, \ol{K}^{\times}) \ar[r] \ar[d]_{\Res}
	& \bigoplus_{v \nmid \infty} \Q/\Z \oplus \bigoplus_{\text{$v$: real}} \frac{1}{2}\Z/\Z \ar[r]^-{\sum} \ar[d]
	& \Q/\Z \ar[r] \ar[d]
	& 0\\
	0 \ar[r]
	& H^2(K', \ol{K'}^{\times}) \ar[r]
	& \bigoplus_{v' \nmid \infty} \Q/\Z \oplus \bigoplus_{\text{$v'$: real}} \frac{1}{2}\Z/\Z \ar[r]_-{\sum}
	& \Q/\Z \ar[r]
	& 0.
}
\]
Here, $v$ (resp.~$v'$) runs over places of $K$ (resp.~$K'$).
The left vertical arrow is the restriction map, the middle one is induced by multiplication 
by the local degree $[K'_{v'}: K_v]$ at each direct summand, and the right one is induced 
by multiplication by $[K': K]$. By the assumption on the local degrees of $\cK$, the terms 
$(\Q/\Z)[2^{\infty}]$ vanish after taking the direct limit to $\cK$ or beyond.
Therefore, taking the direct limit gives us a diagram
\[
\xymatrix{
	H^2(\cK, \ol{\cK}^{\times})[2^{\infty}] \ar[r]^-{\simeq} \ar[d]_{\Res}
	& {\bigoplus}_{\text{$\sV$: real}}^{\prime} \frac{1}{2}\Z/\Z \ar[d]\\
	H^2(\cK_f, \ol{\cK_f}^{\times})[2^{\infty}] \ar[r]_-{\simeq}
	& {\bigoplus}_{\text{$\sV'$: real}}^{\prime} \frac{1}{2}\Z/\Z.
}
\]
Here, $\sV$ (resp.~$\sV'$) runs over places of $\cK$ (resp.~$\cK_f$) and ${\bigoplus}^{\prime}$ 
denotes the direct limit of the corresponding finite terms.
The right vertical arrow is injective because $\cK_f/\cK$ is unramified at all real places.
This implies that the left one is injective, which results in the claim.
\end{proof}

\begin{claim}\label{claim:12}
We have $H^1(\cG_{\cK}, (\cK_f^{\times})^2) = 0$.
\end{claim}

\begin{proof}
By \eqref{eq:2} and Hilbert 90, we only have to show that the natural map
\begin{equation}\label{eq:3}
H^0(\cG_{\cK}, \cK_f^{\times}) \to H^0(\cG_{\cK}, \cK_f^{\times}/(\cK_f^{\times})^2)
\end{equation}
is surjective.

For a number field $K$, we have a natural homomorphism
$$ K^{\times}/(K^{\times})^2 
  \to \bigoplus_{\text{$v$: real}} K_v^{\times}/(K_v^{\times})^2
  \simeq \bigoplus_{\text{$v$: real}} \Z/2\Z.  $$
It is easy to see that this is surjective (e.g., because $K$ is dense in 
$\R \otimes_{\Q} K \simeq \prod_{v \mid \infty} K_v$).
As for the kernel, for an element $x \in K^{\times}$ that is positive at all real places, 
the extension $K(\sqrt{x})/K$ is unramified at all real places.
Therefore, taking the direct limit with $K$ running over the subfields of $\cK_f$, the kernel vanishes.
As a result, we obtain an isomorphism
$$ \cK_f^{\times}/(\cK_f^{\times})^2
  \simeq {\bigoplus_{\text{$\sV'$: real}}}^{\prime \,} \Z/2\Z,  $$
where $\sV'$ runs over the real places of $\cK_f$ and ${\bigoplus}^{\prime}$ 
means the direct limit of the corresponding finite terms.

Now the map \eqref{eq:3} can be identified with
$$ \cK^{\times} 
  \to {\bigoplus_{\text{$\sV$: real}}}^{\prime \,} \Z/2\Z, $$
where $\sV$ runs over the real places of $\cK$.
Then the surjectivity follows from the above-mentioned surjectivity for any $K$.
\end{proof}

This completes the proof of Proposition \ref{prop:cd_1}.
\end{proof}

\section*{Acknowledgments}

The second author is supported by JSPS KAKENHI Grant Number 22K13898.


\end{document}